\documentclass[a4paper,11pt]{amsart}
\usepackage{amssymb}
\usepackage{amscd}

\newtheorem{theorem}{Theorem}[section]
\newtheorem{lemma}[theorem]{Lemma}

\newtheorem{proposition}[theorem]{Proposition}
\newtheorem{corollary}[theorem]{Corollary}

\theoremstyle{definition}
\newtheorem{definition}[theorem]{Definition}

\theoremstyle{remark}
\newtheorem{remark}[theorem]{Remark}

\usepackage{graphicx}
\usepackage{amsmath}
\usepackage{amsfonts}

\let\phi\varphi

\newcommand{\Ga}{\Gamma}
\newcommand{\Om}{\Omega}
\newcommand{\Ph}{\Phi}
\newcommand{\ph}{\phi}
\newcommand{\ps}{\psi}

\newcommand{\om}{\omega}

\newcommand{\al}{\alpha}

\newcommand{\La}{\Lambda}

\newcommand{\si}{\sigma}
\newcommand{\ka}{\kappa}

\newcommand{\eps}{\varepsilon}

\newcommand\fg{\mathfrak{g}}

\newcommand{\ba}{\begin{array}}

\newcommand{\ea}{\end{array}}

\newcommand{\beq}{\begin{eqnarray}}

\newcommand{\eeq}{\end{eqnarray}}

\newtheorem{lm}{lemma}

\newtheorem{thee}{theorem}

\newtheorem{proo}{proposition}

\newtheorem{co}{corollary}

\newtheorem{rem}{remark}

\newtheorem{deff}{definition}

\newcommand{\bd}{\begin{deff}}

\newcommand{\ed}{\end{deff}}

\newcommand{\bl}{\begin{lm}}

\newcommand{\el}{\end{lm}}

\newcommand{\bp}{\begin{proo}}

\newcommand{\ep}{\end{proo}}

\newcommand{\bt}{\begin{thee}}

\newcommand{\et}{\end{thee}}

\newcommand{\bc}{\begin{co}}

\newcommand{\ec}{\end{co}}

\newcommand{\brm}{\begin{rem}}

\newcommand{\erm}{\end{rem}}

\usepackage{t1enc}
\def\frak{\mathfrak}
\def\Bbb{\mathbb}
\def\Cal{\mathcal}

\newcommand{\newc}{\newcommand}

\renewcommand{\exp}{\operatorname{exp}}
\newcommand{\id}{\operatorname{id}}

\newcommand{\Fl}{\operatorname{Fl}}
\newcommand{\tr}{\operatorname{tr}}

\renewcommand{\o}{\circ}

\let\ccdot\cdot
\def\cdot{\hbox to 2.5pt{\hss$\ccdot$\hss}}

\newc{\aR}{\mbox{\boldmath{$ R$}}}
\newc{\aS}{\mbox{\boldmath{$ S$}}}
\newc{\aT}{\mbox{\boldmath{$ T$}}}
\newc{\aW}{\mbox{\boldmath{$ W$}}}

\newc{\aK}{\mbox{\boldmath{$ K$}}}
\newc{\aL}{\mbox{\boldmath{$ L$}}}

\usepackage{amssymb}
\usepackage{amscd}

\newcommand{\nd}{\nabla}
\renewcommand{\and}{\boldsymbol{\nd}}

\let\i=\iota

\newcommand{\nn}[1]{(\ref{#1})}

\newc{\obstrn}[2]{B^{#1}_{#2}}

\newcommand{\rpl}                         
{\mbox{$
\begin{picture}(12.7,8)(-.5,-1)
\put(0,0.2){$+$}
\put(4.2,2.8){\oval(8,8)[r]}
\end{picture}$}}

\newcommand{\lpl}                         
{\mbox{$
\begin{picture}(12.7,8)(-.5,-1)
\put(2,0.2){$+$}
\put(6.2,2.8){\oval(8,8)[l]}
\end{picture}$}}

\usepackage{ifthen}

\newc{\tensor}[1]{#1}
\newc{\Mvariable}[1]{\mbox{#1}}
\newc{\down}[1]{{}_{#1}}
\newc{\up}[1]{{}^{#1}}

\newc{\JulyStrut}{\rule{0mm}{6mm}}
\newc{\midtenPan}{\mbox{\sf S}}
\newc{\midten}{\mbox{\sf T}}
\newc{\midtenEi}{\mbox{\sf U}}
\newc{\ATen}{\mbox{\sf E}}
\newc{\BTen}{\mbox{\sf F}}
\newc{\CTen}{\mbox{\sf G}}

\def\sideremark#1{\ifvmode\leavevmode\fi\vadjust{\vbox to0pt{\vss
 \hbox to 0pt{\hskip\hsize\hskip1em
 \vbox{\hsize3cm\tiny\raggedright\pretolerance10000
 \noindent #1\hfill}\hss}\vbox to8pt{\vfil}\vss}}}%
                        
\numberwithin{equation}{section}

\begin{document}

\markboth{A.\ \v Cap, A.R.\ Gover and M.\ Hammerl}
{Normal BGG solutions and polynomials}

\title{Normal BGG Solutions and Polynomials}

\author{A.\ \v Cap}

\address{
Faculty of Mathematics\\
University of Vienna\\
Nordbergstr. 15\\
1090 Wien\\
Austria\\
Andreas.Cap@univie.ac.at}

\author{A.R.\ Gover}
\address{
Department of Mathematics\\
  The University of Auckland\\
  Private Bag 92019\\
  Auckland 1142\\
  New Zealand;\\
Mathematical Sciences Institute\\
Australian National University \\ ACT 0200, Australia\\
r.gover@auckland.ac.nz}

\author{M.\ Hammerl}
\address{
Department of Mathematics and Computer Science \\
Ernst Moritz Arndt University of Greifswald \\
Walther-Rathenau-Stra{\ss}e 47 \\
17487 Greifswald, Germany \\
matthiasrh@gmail.com
}

\begin{abstract}
  First BGG operators are a large class of overdetermined linear
  differential operators intrinsically associated to a parabolic
  geometry on a manifold. The corresponding equations include those
  controlling infinitesimal automorphisms, higher symmetries, and many
  other widely studied PDE of geometric origin. The machinery of BGG
  sequences also singles out a subclass of solutions called normal
  solutions. These correspond to parallel tractor fields and hence to
  (certain) holonomy reductions of the canonical normal Cartan
  connection. Using the normal Cartan connection, we define a special
  class of local frames for any natural vector bundle associated to a
  parabolic geometry. We then prove that the coefficient functions of
  any normal solution of a first BGG operator with respect to such a
  frame are polynomials in the normal coordinates of the parabolic
  geometry. A bound on the degree of these polynomials in terms of
  representation theory data is derived.

  For geometries locally isomorphic to the homogeneous model of the
  geometry we explicitly compute the local frames mentioned above.
  Together with the fact that on such structures all solutions are
  normal, we obtain a complete description of all first BGG solutions
  in this case.  Finally, we prove that in the general case the
  polynomial system coming from a normal solution is the pull-back of
  a polynomial system that solves the corresponding problem on the
  homogeneous model. Thus we can derive a complete list of potential
  normal solutions on curved geometries. Moreover, questions concerning
  the zero locus of solutions, as well as related finer geometric and
  smooth data, are reduced to a study of polynomial systems and real
  algebraic sets.
\end{abstract}



\maketitle

\section{Introduction}\label{1}

It has long been known that certain natural overdetermined linear
partial differential equations play a fundamental role in differential
geometry. Archetypal examples in Riemannian geometry are the Killing
equation and the conformal Killing equation on various types of tensor
fields. In particular, solutions to these equations on vector fields
are infinitesimal isometries and infinitesimal conformal
isometries. The conformal Killing equation on differential forms (see
e.g.~\cite{semmelmann,leitner-normal}) and the twistor equation, which
is the analogous equation on spinors (see
e.g.~\cite{friedrich-conformalrelation,baum-friedrich-twistors}), have
been intensively studied. More recently, it has been shown that
solutions to the conformal Killing equation on symmetric tensor fields
are equivalent to higher symmetries of the Laplacian, see
\cite{eastwood-laplacian}, and generalisations to other operators have
been obtained, see \cite{gover-silhan-powers}.

For all these equations, many questions arise concerning the nature
and properties of solutions.  For example the question of establishing
the structure of the solution's zero locus has been taken up for
specific equations in many places, see
e.g. \cite{belgun-moroianu-ornea,Derdzinski,habermann-twizero,kuehnel-rademacher-twizero}.
Understanding the structure of solutions is also important because of
a second line of applications of these operators. Aside from
determining notions of symmetry, it is becoming evident that they can
be used to define or characterise important geometric structures. For
example the Poincar\'e-Einstein structures of \cite{fefferman-graham},
which form the basis of a geometric Poisson transform programme
relating conformal and Riemannian geometry \cite{graham-lee-91} as
well as a related scattering programme \cite{graham-zworski}, may be
understood as a conformal manifold equipped with a solution to the
overdetermined equation controlling the conformal-to-Einstein
condition \cite{gover-dirichlet,gover-aes}. In the study of this
description in these the last references, the conformal tractor
connection was used to classify the zero loci of solutions to the
conformal to Einstein equation in a way which highlighted links with
the homogeneous model. This motivated many of the developments in the
current article.

All the equations alluded to above can be understood as special cases
of first BGG equations on parabolic geometries, and this observation
alone suggests powerful generalisations, see
e.g. \cite{cap-infinitaut}. For a semisimple Lie group $G$ with Lie
algebra $\frak{g}$ and a parabolic subgroup $P\subset G$, a parabolic
geometry of type $(G,P)$ is an $n$-manifold $M$ equipped with a
$P$--principal bundle $\mathcal{G} \to M$ endowed with a (suitably
\textit{normal}) \textit{Cartan connection}
$\om\in\Om^1(\mathcal{G},\frak{g})$ (see Definition
\ref{def:par-geom}). The homogeneous model of the geometry is the
``Klein structure'' consisting of $G$ viewed as a $P$-principal bundle
over the homogeneous space $G/P$.  Parabolic geometries form a broad
class of structures which includes conformal, CR, and projective
geometries as special cases. For example in the case of signature
$(p,q)$ conformal geometry (with $p+q=n$), $G=SO(p+1,q+1)$ and
(identifying $G$ with its linear defining representation) $P$ is the
parabolic stabilising a nominated null ray.

On any parabolic geometry of type $(G,P)$ the normal Cartan connection
universally determines, for each irreducible representation
$\mathbb{V}$ of $G$, a canonical sequence of differential operators
with the property that on the model $G/P$ this sequence is a finite
resolution of $\mathbb{V} $ by linear differential operators
\cite{CSSBGG}. These sequences are known as BGG sequences, due to the
relations to the algebraic resolutions of Bernstein-Gelfand-Gelfand
and others \cite{BGG,lepowsky}. The first operator in each such
sequence is called a first BGG operator, its equation the first BGG
equation, and ranging over the possible $(G,P,\mathbb{V})$ we obtain
the class mentioned above.

First BGG equations have finite dimensional kernel and they have no
non-trivial solutions in general. This is in fact part of their
importance: the integrability conditions for the existence of
non-trivial solutions are often extremely interesting non-linear
systems intrinsic to the structure (such as the Einstein equation in
certain cases). Normal (first) BGG solutions are a distinguished class
of solutions; the solutions in the class are characterised by the fact
that they correspond in a precise way to a holonomy reduction of the
normal Cartan connection (see e.g.\ \cite{leitner-normal}).  The
correspondence is mediated by what is called the tractor connection;
this linear connection is induced by the Cartan connection on the
bundle associated to $\mathcal{G}$ via the relevant irreducible
$G$-representation $\mathbb{V}$ (see Proposition
\ref{prop:firstbgg}). An important point is that on any homogeneous
model $G/P$, for any $G$-irreducible $\mathbb{V}$ all solutions of the
corresponding first BGG operator are normal and the vector space of
such solutions is isomorphic to $\mathbb{V}$.

For the special case of projective structures, we have developed in
\cite{CGH-projective} a new approach to the study of first BGG
solutions. In particular, that article discusses applications to the
construction and study of new compactifications, which emphasise the
geodesic structure rather than conformal aspects. The developments in
that article exhibited two main features, the polynomiality of normal
solutions, and the comparison between curved geometries and the
homogeneous model, which looks particularly simple on the level of
tractor bundles. It turns out that the latter aspect can be understood
in a conceptual way in the context of holonomy reductions of general
Cartan geometries. This is developed in \cite{CGH-holonomy}, which may
be viewed as a companion article to the current work, with parabolic
geometries providing the most important examples. It turns out that
holonomy reductions of parabolic geometries determined by normal
solutions of first BGG operators govern a host of interesting
constructions which relate apparently different geometries, such as
\cite{bryant-3planes,fefferman,nurowski-metric,mrh-sag-twistors}. 

The main aim of the current article is to extend the polynomiality
results from \cite{CGH-projective} to normal first BGG solutions on
general parabolic geometries. This leads to remarkably strong results
on the nature of these solutions. Given a point in the manifold, we
first have to choose a point in the fibre of the Cartan bundle, so
there is a freedom parametrised by the parabolic subgroup $P$. Having
made this choice, we get local normal coordinates as well as a special
class of local frames for all natural vector bundles associated to the
parabolic geometry in question. We then prove that the coefficient
functions of any normal solution with respect to such a frame are
polynomials in the normal coordinates. The degree of these polynomials
is known a priori in terms of data associated with the equation. These
results form the main theorem, which is Theorem \ref{thm:basic}.  The
proof of that theorem is constructive: the polynomial is produced
directly from the parallel tractor field via a preferred local
trivialisation of the tractor bundle that we also describe.

As mentioned above, via the parallel tractor field corresponding to a
normal solution of a first BGG operator, there is a connection to
holonomy reductions of Cartan geometries as studied in
\cite{CGH-holonomy}. One of the main features of such a reduction is an
induced stratification of the underlying manifold into initial
submanifolds, which reflects a certain orbit decomposition of the
homogeneous model. The zero locus of the corresponding normal solution
is encoded in this data (see Section \ref{types} below); in fact, the
full stratification reveals that considerable further information is
available. Understanding this stratification from a more analytic
point of view requires a description by functions and this is
precisely the information captured by the polynomial systems we
discover and describe here. Indeed, our results may be viewed as
showing that the submanifolds arising can be understood as
generalising, in a natural way, a class of projective algebraic sets.

In the case of the homogeneous model $G/P$ we give an explicit
description of the special frames on tensor bundles in Proposition
\ref{prop:homog-frame}, thus obtaining a completely explicit
description of all solutions of the first BGG equations on such
bundles. (Then the same result holds locally on any structure which is
locally flat, i.e.~locally isomorphic to the model). A second
important, and surprising, consequence of the construction is that on
any curved parabolic geometry of type $(G,P)$ the polynomial system
describing a normal solution is actually the pull-back, via a special
diffeomorphism, of the polynomial system describing a normal solution
(of the same first BGG equation) on $G/P$, see Proposition
\ref{prop:comparison}. This diffeomorphism comes from a comparison map
that we construct in Definition \ref{def:comp}. On the homogeneous
space $G/P$ the structures defined by the polynomial system can often
be well understood via classical techniques (or are even well known),
and our result here gives a precise statement to which extent these
features must hold for a general normal solution.

Although the general theory is simple and universal, each particular
case of a normal solution of a given first BGG equation on a specific
parabolic geometry typically carries considerable information and
detail.  In Section \ref{appl} we illustrate how this can be obtained
in a completely explicit way, and how it yields applications in
familiar settings. Examples include conformal first BGG equations on
densities, conformal Killing vector fields, conformal Killing
$r$-forms, Killing vector fields, as well as several examples from
projective geometry.  In addition we show in that section how the
generalised homogeneous projective coordinates, discovered in
\cite{CGH-projective}, arise using the very different perspective we
develop here, see Proposition \ref{prop:proj1}. This then enables the
corresponding notion to be developed for conformal geometry, which is
formalised in Definition \ref{def:conf-homog}.

\section{Normal BGG solutions}\label{2}
We start by very briefly recalling the algebraic background and the
definition of parabolic geometries, referring to Section 3.1 of
\cite{cap-slovak-book} for details.

\subsection{$|k|$--graded Lie algebras and their
  representations}\label{2.1} The starting point for defining a
parabolic geometry is a real or complex semisimple Lie algebra $\frak
g$ endowed with a so--called $|k|$--grading, i.e.~a decomposition
$\fg=\fg_{-k}\oplus\dots\oplus\fg_k$ such that
$[\fg_i,\fg_j]\subset\fg_{i+j}$ for some $k\geq 1$. We make the
standard assumptions that none of the simple ideals of $\fg$ is
contained in $\fg_0$ and that the subalgebra
$\fg_-:=\fg_{-k}\oplus\dots\oplus\fg_{-1}$ is generated by
$\fg_{-1}$. In particular, this implies that $\fg_0\subset\fg$ is a
subalgebra which acts on each $\fg_i$ via the restriction of the
adjoint representation.

Next, we consider the associated filtration
$\fg=\fg^{-k}\supset\fg^{-k+1}\supset\dots\supset\fg^k$ defined by
$\fg^i:=\fg_i\oplus\dots\oplus\fg_k$. This makes $\fg$ into a filtered
Lie algebra in the sense that $[\fg^i,\fg^j]\subset\fg^{i+j}$. In
particular, this implies that $\frak p:=\fg^0$ is a subalgebra of
$\fg$, which acts on each filtration component $\fg^i$ via the
restriction of the adjoint action. In particular, $\frak p_+=\fg^1$ is
an ideal in $\frak p$, which is nilpotent by definition. It turns out
that the subalgebra $\frak p\subset\frak g$ is always a parabolic
subalgebra in the sense of representation theory, and that any
parabolic subalgebra can be realised via an appropriate
$|k|$--grading. It further turns out that $\frak p_+\subset\frak p$ is
the nilradical of $\frak p$ and $\frak g_0$ is a reductive complement
to $\frak p_+$ in $\frak p$, which is usually called a Levi--factor.

In this article, we will only consider finite dimensional
representations.  Any such representation $\Bbb V$ of $\fg$ inherits a
grading, which is compatible with the grading on $\fg$. It can be
shown that the grading of $\fg$ is the eigenspace decomposition with
respect to the adjoint action of a uniquely determined element
$E\in\frak g$, called the \textit{grading element}. This element then
also acts diagonalisably on each irreducible representation $\Bbb V$
of $\fg$ and one can interpret the eigenspace decomposition as a
decomposition $\Bbb V=\Bbb V_0\oplus\dots\oplus\Bbb V_N$ such that
$\fg_i\cdot\Bbb V_j\subset\Bbb V_{i+j}$. (In many cases it is more
natural to write this decomposition in the form $\Bbb
V_{-\ell}\oplus\dots\oplus\Bbb V_{\ell}$, but one can always shift the
degrees to obtain the form described above.)

\subsection{Parabolic geometries and BGG sequences}\label{2.2}

Take a $|k|$--graded Lie algebra $\fg$ as in Section \ref{2.1} and a
Lie group $G$ with Lie algebra $\fg$. Then it can be shown that the
normaliser $N_G(\frak p)$ of $\frak p$ in $G$ has Lie algebra $\frak
p$, and one makes a choice of a \textit{parabolic subgroup} $P\subset
G$, i.e.~a subgroup lying between $N_G(\frak p)$ and its connected
component of the identity. Then the adjoint action of each element of
$P$ preserves each of the filtration components $\fg^i\subset\fg$, and
one defines the \textit{Levi--subgroup} $G_0\subset P$ to consist of
all elements whose adjoint action preserves the grading of $\fg$. One
shows that $G_0\subset P$ corresponds to the Lie subalgebra
$\fg_0\subset\frak p$. Moreover, the exponential mapping defines a
diffeomorphism from $\frak p_+$ onto a closed normal subgroup
$P_+\subset P$ such that $P$ is the semi--direct product of $G_0$ and
$P_+$.

\begin{definition}\label{def:par-geom}
(1) A \textit{parabolic geometry} of type $(G,P)$ on a smooth manifold
  $M$ is given by a principal $P$--bundle $p:\Cal G\to M$ endowed with
  a Cartan connection $\om\in\Om^1(\Cal G,\frak g)$. By definition,
  this means that for each $u\in\Cal G$ the map $\om_u:T_u\Cal
  G\to\frak g$ is a linear isomorphism, that $\om$ is equivariant with
  respect to the principal right action of $P$, and reproduces the
  generators of fundamental vector fields.

(2) The \textit{homogeneous model} for parabolic geometries of type
  $(G,P)$ is the homogeneous space $G/P$ (which is a generalised flag
  manifold) with the left Maurer--Cartan form as the Cartan
  connection.
\end{definition}

Parabolic geometries which satisfy the additional conditions of
regularity and normality are equivalent encodings (in the categorical
sense) of certain underlying structures, see Section 3.1 of
\cite{cap-slovak-book} and Chapter 4 of this reference for many
examples. These underlying structures include classical projective
structures, conformal structures, almost quaternionic structures,
hypersurface type CR structures, and several types of generic
distributions. A congenial feature of the theory we develop below is
that it does not depend at all on the details of the correspondence to
underlying structures, nor does it need explicit details of the
precise definitions of regularity and normality. Thus we shall
simply take the parabolic geometry as an input and, for the general
results (which form the main theorems), this results in an efficient
simultaneous treatment of the entire class of parabolic geometries.

Forming associated bundles to the Cartan bundle, any representation of
the Lie group $P$ gives rise to a natural bundle on parabolic
geometries of type $(G,P)$. In particular, one can use the
restrictions to $P$ of (irreducible) representations of $G$, thus
obtaining so--called \textit{tractor bundles}. While these bundles may
at first seem unusual as geometric objects, they have the critical
advantage that (unlike general associated bundles) they inherit a
canonical linear connection, called the \textit{tractor connection}
from the Cartan connection $\om$. The general theory of tractor
bundles is developed in \cite{cap-gover-tractor}. More conventional
geometric objects like tensors and spinors are obtained by forming
so--called \textit{completely reducible} bundles, which are associated
to completely reducible representations of $P$. It is well-known that
these representations are obtained from completely reducible
representations of $G_0$ via the projection $P\to G_0=P/P_+$.

The tractor connections can be used to construct higher order
differential operators acting between completely reducible bundles
which are intrinsic to the geometric structure via the machinery of
BGG sequences, which was introduced in \cite{CSSBGG} and improved in
\cite{BGG-Calderbank-Diemer}. For our current purposes, we just need a
small amount of information concerning the first operator in such a
sequence. Given an irreducible representation $\Bbb V$ of $\fg$
consider the grading $\Bbb V=\Bbb V_0\oplus\dots\oplus\Bbb V_N$ as in Section
\ref{2.1} and the induced filtration $\{\Bbb V^i\}$ of $\Bbb V$
defined by $\Bbb V^i=\Bbb V_i\oplus\dots\oplus\Bbb V_N$. Then each of
the filtration components $\Bbb V^i$ is $P$--invariant and the induced
action of $P_+$ on $\Bbb V^i/\Bbb V^{i+1}$ is trivial, so this is a
completely reducible representation of $P$. In particular, this
applies to the quotient $\Bbb V/\Bbb V^1=:\Bbb H_0$, which is
immediately seen to be even an irreducible representation of $P$. From
the construction it is clear that $\Bbb H_0$ is obtained from the
$G_0$--representation $\Bbb V_0$ via the quotient projection $P\to
G_0$.

Passing to associated bundles, we write $\Cal VM:=\Cal G\times_P\Bbb
V$. The $P$--invariant filtration $\{\Bbb V^i\}$ of $\Bbb V$
corresponds to a filtration of $\Cal VM$ by smooth subbundles $\Cal
V^iM$. We then have the irreducible quotient bundle $\Cal H_0M:=\Cal
VM/\Cal V^1M$ and a natural projection $\Pi:\Cal VM\to\Cal H_0M$. The
machinery of BGG sequences constructs a completely reducible
subquotient $\Cal H_1M$ of the bundle $T^*M\otimes\Cal VM$ of $\Cal
VM$--valued one--forms and a natural differential operator $D^{\Cal
  V}:\Ga(\Cal H_0M)\to\Ga(\Cal H_1M)$, called the \textit{first BGG
  operator} associated to $\Bbb V$. Given $\Bbb V$, the
representations $\Bbb H_0$ and $\Bbb H_1$ can be determined
algorithmically, which also leads to a description of the order and
the principal part of the operator $D^{\Cal V}$. Simple algorithms for
the formulae for the full operators $D^{\Cal V}$ are available and
explicit formulae in low order are explicitly available
\cite{calderbank-diemer-soucek,gover-89}. Let us collect the main
information we will need in the sequel. Proofs of the following facts
can be found in \cite{CSSBGG}.

\begin{proposition}\label{prop:firstbgg}
Let us denote by $\Pi:\Cal VM\to\Cal H_0M$ the canonical projection
from a tractor bundle to its canonical quotient as well as the induced
operator on sections.

(1) The kernel of $D^{\Cal V}$ is always finite dimensional with
dimension bounded by $\dim(\Bbb V)$. 

(2) If $s\in\Ga(\Cal VM)$ is parallel for the tractor connection on
$\Cal VM$, then $\Pi(s)\in\Ga(\Cal H_0)$ lies in the kernel of
$D^{\Cal V}$. 

(3) The restriction of $\Pi$ to the space of sections which are
parallel for the tractor connection is injective. 

(4) In the case of the homogeneous model $G/P$ the tractor connection
is flat with trivial holonomy and $\Pi$ defines a linear isomorphism
from the space of parallel sections of $\Cal V(G/P)$ (which can be
identified with $\Bbb V$) onto the kernel of $D^{\Cal V}$. 
\end{proposition}

\begin{definition}\label{def:norm-sol}
(1) The \textit{first BGG equation} determined by $\Bbb V$ is the
  natural differential equation $D^{\Cal V}(\al)=0$ on $\al\in\Ga(\Cal
  H_0M)$.

(2) A solution of the first BGG equation determined by $\Bbb V$
  is called \textit{normal} if it is of the form $\Pi(s)$ for a
  parallel section $s$ of the tractor bundle $\Cal VM$. 
\end{definition}

Note that part (4) of Proposition \ref{prop:firstbgg} says that on the
homogeneous model (and hence also on geometries locally isomorphic to
the homogeneous model) any solution of a first BGG equation is
normal. It should be remarked at this point, that the name ``normal''
reflects that the solutions concerned correspond to parallel sections
of the normal (in the sense of \cite{cap-gover-tractor}) tractor
connection, and this terminology is already established in the
literature (cf.\ \cite{leitner-normal}). It is not directly related to
  normal coordinates and normal frames which we will discuss next.

\subsection{Normal coordinates and normal frames}\label{nframes}
We now come to the basic construction needed to describe parallel
sections of tractor bundles and hence normal solutions of first BGG
operators. 

Let us start with a parabolic geometry $(p:\Cal G\to M,\om)$ of some
fixed type $(G,P)$. Fix a point $u_0\in\Cal G$ and put $x_0:=p(u_0)\in
M$. Consider the subalgebra $\fg_-\subset\fg$ which is complementary
to $\frak p$ as in \ref{2.1}. For any $X\in\fg_-$ we can consider the
constant vector field $\tilde X\in\frak X(\Cal G)$ which is
characterised by $\om(\tilde X)(u)=X$ for all $u\in\Cal G$. There is
an open neighbourhood $V\subset\fg_-$ of zero such that the flow
$\Fl^{\tilde X}_t(u_0)$ through $u_0$ is defined up to time $t=1$ for
all $X\in V$. Then $\Ph(X):=\Fl^{\tilde X}_1(u_0)$ defines a smooth
map $\Ph:V\to \Cal G$ and we define $\ph:=p\o\Ph:V\to M$.

By construction, $\ph(0)=x_0$ and the derivative $T_0\ph:\fg_-\to
T_{x_0}M$ at this point is given by $X\mapsto
T_{u_0}p\cdot\om_{u_0}^{-1}(X)$. Since $\fg_-$ is complementary to
$\frak p$, this is a linear isomorphism by the defining properties of
a Cartan connection. Hence we can shrink $V$ in such a way that $\ph$
defines a diffeomorphism from $V$ onto an open neighbourhood $U$ of
$x_0$ in $M$.

\begin{definition}
(1) The \textit{normal chart} determined by $u_0$ is the
  diffeomorphism $\ph^{-1}:U\to V\subset \fg_-$. Choosing a basis in
  $\fg_-$, we get induced local coordinates on $M$ called the
  \textit{normal coordinates} determined by $u_0$.

(2) The \textit{normal section} of  $\Cal G|_U$ determined by $u_0$ is
  the smooth map $\si:U\to \Cal G$ characterised by
  $\si(p(\Ph(X))=\Ph(X)$ for all $X\in V$. 
\end{definition}

\begin{remark}
One can slightly vary the construction of the normal section as
follows. Since $\Cal G\to M$ is a principal $P$--bundle, any closed
subgroup of $P$ acts freely on $\Cal G$ via the principal right
action. Applying this to the closed subgroup $P_+\subset P$ from
\ref{2.1}, one shows that $\Cal G_0:=\Cal G/P_+$ is a smooth manifold,
$p:\Cal G\to M$ descends to a map $p_0:\Cal G_0\to M$ and this is a
principal bundle with structure group $P/P_+=G_0$. There is an obvious
projection $q :\Cal G\to\Cal G_0$ which is actually a principal
bundle with structure group $P_+$. 
As we have seen,
$p_0\o q\o\Ph=p\o\Ph$ is a diffeomorphism, so $ q\o\Ph$ meets any
fibre of $\Cal G_0$ in at most one point. Hence the map sending
$ q(\Ph(X))$ to $\Ph(X)$ naturally extends to a $G_0$--equivariant
smooth section of $ q$ over $(p_0)^{-1}(U)$. In the language of
section 5.1.12 of \cite{cap-slovak-book}, this is the \textit{normal Weyl
  structure} centred at $x_0$ which is determined by $u_0\in
p^{-1}(x_0)$. In this way, one gets additional data, e.g.~the so
called Weyl connection on associated bundles, but we will not need
those in the sequel. 
\end{remark}

\medskip

Recall next, that the natural vector bundles for a parabolic geometry
of type $(G,P)$ are the associated vector bundles to the Cartan
bundle, so they are equivalent to representations of the parabolic
subgroup $P$. Likewise, natural bundle maps between such bundles are
equivalent to $P$--equivariant maps between the inducing
representations. Of course, a local section of a principal bundle
gives rise to a local trivialisation and hence also to local
trivialisations of all associated bundles. 

\begin{definition}\label{def:normtriv}
(1) Let $(p:\Cal G\to M,\om)$ be a parabolic geometry of type $(G,P)$,
  $u_0\in\Cal G$ a point and $\Bbb W$ a representation of $P$. 

The \textit{normal trivialisation} of the associated bundle $\Cal
WM=\Cal G\times_P\Bbb W$ determined by $u_0$ is the trivialisation induced
by the normal section determined by $u_0$. 

(2) A \textit{normal frame} for $\Cal WM$ determined by $u_0$ is a
frame obtained from a basis of $\Bbb W$ via a normal trivialisation. 
\end{definition}

Explicitly, given a representation $\Bbb W$ of $P$, the corresponding
natural vector bundle $\Cal WM$ is $\pi:\Cal G\times_P\Bbb W\to M$, where
$\Cal G\times_P\Bbb W$ is the quotient of $\Cal G\times\Bbb W$ by the action
of $P$ defined by $(u,w)\cdot b:=(u\cdot b,b^{-1}\cdot w)$. Here we
use the principal right action in the first component and the given
representation of $P$ on $\Bbb W$ in the second component. The
trivialisation induced by $\si$ is then given by the map $U\times\Bbb
W\to\pi^{-1}(U)$ which maps $(x,w)$ to the $P$--orbit of
$(\si(x),w)$. Conversely, given $(u,w)\in\Cal G\times\Bbb W$ with
$x=p(u)\in U$, there is a unique element $b\in P$ such that
$u=\si(x)\cdot b$, so the $P$--orbit of $(u,w)$ corresponds to
$(x,b\cdot w)$ in the trivialisation. By construction, the natural
bundle map induced by a $P$--equivariant map $\al:\Bbb W\to\Bbb W'$
corresponds to $\id_U\times\al:U\times\Bbb W\to U\times\Bbb W'$ in this
trivialisation.

It will also be very useful in the sequel to have a description of the
trivialisations determined by $\si$ in terms of smooth
sections. Recall that smooth sections of the bundle $\Cal G\times_P\Bbb
W\to M$ over $U\subset M$ are in bijective correspondence with smooth
maps $f:p^{-1}(U)\to\Bbb W$, which are $P$--equivariant in the sense
that $f(u\cdot b)=b^{-1}\cdot f(u)$. In the local trivialisation
determined by $\si$, this section is simply given by $x\mapsto
(x,f(\si(x)))$, so it simply corresponds to the function $f\o\si:
U\to\Bbb W$.

Since natural bundle maps are induced by equivariant maps between the
representations inducing the bundles, they are nicely compatible with
normal trivialisations and with normal frames. Let us just formulate a
particularly important case explicitly.

\begin{lemma}\label{lem:adapted}
Let $\Bbb W$ and $\Bbb W''$ be representations of $P$ and let $\Bbb
W'\subset\Bbb W$ be a $P$--invariant subspace such that $\Bbb W/\Bbb
W'\cong\Bbb W''$, and let $n'$ and $n''$ be the dimensions of $\Bbb
W'$ and $\Bbb W''$, respectively. Let $(p:\Cal G\to M,\om)$ be a
parabolic geometry of type $(G,P)$, and let $\tau:\Cal WM/\Cal
W'M\to\Cal W''M$ be the induced isomorphism of associated bundles.

Then for any point $u_0\in\Cal G$ there is a normal frame of $\Cal WM$
over the corresponding subset $U$ which has the form
$\{s_1,\dots,s_{n'},s_{n'+1},\dots,s_{n'+n''}\}$ where
$\{s_1,\dots,s_{n'}\}$ is a normal frame for $\Cal W'M$ over $U$ and
$\{\tau(s_{n'+1}),\dots,\tau(s_{n'+n''})\}$ is a normal frame for $\Cal
W''M$ over $U$.
\end{lemma}
\begin{proof}
We just have to choose a basis of $\Bbb W'$ and extend it to a basis
of $\Bbb W$. The additional elements will then descend to a basis of
$\Bbb W/\Bbb W'$, so there is a corresponding basis of $\Bbb
W''$. Then the normal frames of the three induced bundles determined
by these three bases will be related in the way claimed in the lemma.
\end{proof}

\subsection{The case of the homogeneous model}\label{homog}
The homogeneous model of parabolic geometries of type $(G,P)$ is the
principal $P$--bundle $p:G\to G/P$ together with the (left) Maurer
Cartan form as a Cartan connection. For this example, we can describe
normal sections and normal coordinates in a completely explicit way.
We also obtain an explicit description of the normal frames of the
tangent bundle. By naturality of normal frames, this also provides
normal frames for all tensor bundles. For more complicated examples of
parabolic geometries, which involve a non--trivial filtration of the
tangent bundle, we can obtain normal frames adapted to the filtration
in this way. Via Lemma \ref{lem:adapted}, these can be used to obtain
normal frames of the quotients of subsequent filtration components
which provide the main constituents for natural bundles for such
geometries.

The results we are going to prove in the sequel all describe the
component functions obtained by expanding certain sections in terms of
a normal frame. Together with the description of normal frames we will
prove now, we get a complete understanding of first BGG solutions in
the case of the homogeneous model. To obtain similar results for
general geometries, one needs an explicit description of a normal
frame, which essentially amounts to getting an explicit description of
the canonical Cartan connection.

By homogeneity, we can take the distinguished point in the total space
of the bundle $G\to G/P$ to be the identity element $e$. By
definition, the constant vector field $\tilde X\in\frak X(G)$ for the
Maurer Cartan form generated by $X\in\fg$ is simply the left invariant
vector field $L_X$. Hence $\Fl^{\tilde X}_t(e)=\exp(tX)$,
$\Ph(X)=\exp(X)$ and $\ph(X)=\exp(X)P$ for $X\in\fg_-$. It is well
known that in this case $\ph$ is defined on all of $\fg_-$ and it
defines a diffeomorphism onto an open subset of $G/P$. The normal Weyl
structure on this open subset of $G/P$ obtained in this way is the
\textit{very flat Weyl structure} as described in Example 5.1.12 of
\cite{cap-slovak-book}.

To formulate the description of normal frames for the tangent bundle,
recall that there is a unique connected and simply connected Lie group
$G_-$ with Lie algebra $\frak g_-$. This is often called the
\textit{Carnot--group} associated to $\fg_-$. It is well known that
the exponential map for this group defines a global diffeomorphism
$\exp_{G_-}:\fg_-\to G_-$. Now any vector $X\in\fg_-$ generates a left
invariant vector field on $G_-$. In particular, choosing a basis for
$\fg_-$ we obtain a global left invariant frame for the tangent bundle
$TG_-$.

\begin{proposition}\label{prop:homog-frame}
In the normal chart determined by $e\in G$, the normal frames for the
tangent bundle $T(G/P)$ are exactly the pullbacks along the
diffeomorphism $\exp_{G_-}:\fg_-\to G_-$ of the left invariant frames
of the tangent bundle $TG_-$. 

In particular, if $\fg_-$ is abelian, then these are exactly the
holonomic frames of coordinate vector fields determined by a choice of
basis of $\fg_-$.  
\end{proposition}
\begin{proof}
Given a Cartan geometry $(p:\Cal G\to M,\om)$ of type $(G,P)$ an
element $X\in\fg$ and a point $u\in\Cal G$, one immediately sees that
$T_up\cdot \tilde X(u)\in T_{p(u)}M$ depends only on $X+\frak
p\in\frak g/\frak p$. This easily implies that the tangent bundle $TM$
can be identified with the associated bundle $\Cal G\times_P(\fg/\frak
p)$. To obtain a normal frame for $T(G/P)$, we thus have to choose a
basis of $\frak g/\frak p$, choose preimages of the basis elements in
$\fg$, and then project the corresponding left invariant vector fields
to $G/P$. Since $\frak g=\frak g_-\oplus\frak p$, there are unique
preimages contained in $\fg_-$ and conversely, any basis of $\fg_-$
projects onto a basis of $\frak g/\frak p$. Thus we may as well start
from a basis of $\fg_-$.  

Now the inclusion of $\fg_-$ into $\fg$ uniquely lifts to an injective
homomorphism $i:G_-\to G$. In particular, this implies that the tangent
map $Ti$ maps left invariant vector fields to left invariant vector
fields. Since $\Ph=i\o\exp_{G_-}$, we obtain the first result. 

If $\fg_-$ is abelian, then $G_-=\fg_-$ (with the group structure
given by addition) and $\exp_{G_-}=\id$. Moreover, the left invariant
vector fields for this group are just the constant vector fields for
the obvious trivialisation of $T\fg_-$, i.e.~the coordinate vector
fields.
\end{proof}

\subsection{The basic polynomiality result}\label{basic}
We are now ready to prove the basic fact that the coefficients of any
normal solution to a first BGG operator in a normal frame are
polynomials in the corresponding normal coordinates; we also obtain a 
bound on the degree of these polynomials. First, we need a lemma:

\begin{lemma}\label{lem:basic}
  Let $(p:\Cal G\to M,\om)$ be a parabolic geometry of some fixed type
  $(G,P)$ and consider the tractor bundle $\Cal VM\to M$ corresponding
  to a representation $\Bbb V$ of $G$. Fix a point $u_0\in\Cal G$, write
  $x_0=p(u_0)\in M$ and consider the normal section $\si:U\to\Cal G$
  centred at $x_0$ which is determined by $u_0$.
  
  If $s\in\Ga(\Cal VM)$ is parallel for the canonical tractor
  connection, then the function $f:U\to\Bbb V$, which describes $s$ in
  the given normal trivialisation, is given by $f(\ph(X))=\exp(-X)\cdot
  f(x_0)$.
\end{lemma}
\begin{proof}
  Let $\tilde f:\Cal G\to\Bbb V$ be the $P$--equivariant function
  corresponding to $s$. Denoting the tractor connection by $\nabla^V$,
  the function $\Cal G\to L(\frak g/\frak p,\Bbb V)$ corresponding to
  $\nabla^V s$ is given by $u\mapsto ((X+\frak p)\mapsto
  (\om_u^{-1}(X)\cdot\tilde f)(u)+X\cdot (\tilde f(u)))$. Now consider the
  smooth curve $c(t):=\Fl^{\tilde X}_t(u_0)$, where $\tilde X$ is the
  constant vector field determined by $X\in\fg_-$. Then
  $c'(t)=\om_{c(t)}^{-1}(X)$ for all $t$, so $(\tilde f\o
  c)'(t)=\om_{c(t)}^{-1}(X)\cdot\tilde f$. If the initial section $s$
  is parallel, we conclude that $\tilde f\o c$ satisfies the
  differential equation $(\tilde f\o c)'(t)=-X\cdot (\tilde f\o
  c)(t)$. Since the unique solution of this is $(\tilde f\o
  c)(t)=\exp(-tX)\cdot \tilde f(c(0))$, $c(0)=\si(x_0)$ and for $X$
  close enough to zero we have $c(1)=\si(\ph(X))$ the result follows
  immediately from $f=\tilde f\o\si$.
\end{proof}

Now recall, from Section \ref{2.2}, that given an irreducible
$G$--representation $\Bbb V$ we have the filtration $\{\Bbb V^i\}$ and
the $P$--irreducible quotient $\Bbb H_0$. For a given parabolic
geometry $(p:\Cal G\to M,\om)$ we can then consider normal solutions
of the first BGG operator determined by $\Bbb V$. By definition, these
are the images of those sections of $\Cal VM=\Cal G\times_P\Bbb V$ which are
parallel, for the tractor connection, under the natural bundle map
$\Pi:\Cal VM\to\Cal H_0M$ induced by the quotient projection. Now we
are ready to formulate our basic polynomiality result.

\begin{theorem}\label{thm:basic}
  Let $(p:\Cal G\to M,\om)$ be a parabolic geometry of type $(G,P)$
  and let $\Bbb V$ be a representation of $G$ with natural grading
  $\Bbb V=\Bbb V_0\oplus\dots\oplus\Bbb V_N$, and suppose that
  $\al\in\Ga(\Cal H_0M)$ is a normal solution to the first BGG
  operator determined by $\Bbb V$.
  
  Then for any normal section $\si$, the coefficients of $\al$ in a
  normal frame are polynomials of degree at most $N$ in the normal
  coordinates determined by $\si$.
\end{theorem}
\begin{proof}
  Having given the normal section $\si$, we use normal frames as in
  Lemma \ref{lem:adapted} by choosing an appropriate basis of $\Bbb
  V$. Having chosen this basis, we obtain an identification of $L(\Bbb
  V,\Bbb V)$ with the space of real matrices of appropriate
  size. Likewise, we have to choose a basis of $\fg_-$ in order to
  obtain normal coordinates, and we extend this to a basis of $\frak
  g$.
  
  The representation $\rho $ of $\frak g$ on $\Bbb V$ is a linear map, so for
  $X\in\fg_-$ and $v\in\Bbb V$ the coordinates of $X\cdot v=\rho(X) v$, with
  respect to the chosen basis of $\Bbb V$, are linear expressions in
  the coordinates of $X$ in the chosen basis of $\fg_-$. Now the
  grading $\{\Bbb V_j\}$ of $\Bbb V$ has the property that
  $\fg_i\cdot\Bbb V_j\subset\Bbb V_{i+j}$. In particular, for any
  element $X\in\fg_i$ with $i<0$, the corresponding linear map
  $\rho(X):\Bbb V\to\Bbb V$ has the property that $\rho(X)^{N+1}=0$.

Now let $s\in\Ga(VM)$ be a section which is parallel for the tractor
connection induced by $\om$ and let $f:U\to\Bbb V$ be the function
describing $s$ in the normal trivialisation induced by $\si$. Then from
Lemma \ref{lem:basic} we know that
\begin{equation}\label{poly}
f(\ph(X))=\exp(-X)\cdot v_0=\sum_{k=0}^N\frac{(-1)^k}{k!}\rho(X)^k v_0
\end{equation}
 for an appropriate element $v_0\in\Bbb V$. But this implies that
the components of $s$ in the normal frame determined by $\si$ and the
chosen basis of $\Bbb V$ are polynomials of degree at most $N$ in the
normal coordinates. By construction, the bundle map $\Pi:\Cal VM\to
\Cal H_0M$ maps some elements of the normal frame of $\Cal VM$ to the
normal frame of $\Cal H_0M$ and the remaining elements to zero, so
since $\al=\Pi\o s$, the result follows.
\end{proof}

\begin{remark}
Note that in the proof above we obtain more than is claimed in the
theorem.  The right-hand-side of expression \nn{poly} {\em is} the
polynomial system describing the parallel tractor $s$ in the normal
trivialisation.

Obtaining the explicit presentation of the polynomial system giving
$\alpha$ now simply requires the description of $\Pi$ arising from the
choice of adapted basis for $\mathbb{V}$. This requires an application
of elementary representation theory, the details of which 
depend on $G$, $P$, and $\mathbb{V}$.  For a given parabolic geometry 
(so with $G$, and $P$ fixed)
a practical solution to this,
which we shall take up in section \ref{appl} below, is to
exploit the fact that representations of $G$ can be built up via
tensorial constructions from some basic representations and
correspondingly tractor bundles can be built up from some basic
bundles. 
\end{remark}

Before we discuss examples in detail we continue with a line of
argument that does not depend on knowing the particular parabolic
geometry or details of the representation $\mathbb{V}$, namely the
comparison to the homogeneous model, which is the main tool in the
study of projective structures in \cite{CGH-projective} and of holonomy
reductions of general Cartan geometries in \cite{CGH-holonomy}.

\subsection{Comparison to the homogeneous model}\label{comparison}
Recall the description of the normal section and normal chart
determined by $e\in G$ for the homogeneous model $G\to G/P$ from
section \ref{homog}. From now on, to distinguish this setting from the
general case, we shall write $\underline{\Ph}$ and $\underline{\ph}$
for these maps. Recall that $\underline{\ph}$ is defined on all of
$\fg_-$ and gives a diffeomorphism onto an open subset of $G/P$. Now
we can combine the construction of section \ref{nframes} for a general
geometry of type $(G,P)$ with the same construction for the
homogeneous model.

\begin{definition}\label{def:comp}
Let $(p:\Cal G\to M,\om)$ be a parabolic geometry of type $(G,P)$. The
\textit{normal comparison map} determined by a point $u_0\in\Cal G$
is the composition $\ps:=\underline{\ph}\o\ph^{-1}$, where $\ph^{-1}$
denotes the normal chart determined by $u_0$. 
\end{definition}

Evidently, the normal comparison map defines a diffeomorphism from
the domain $U$ of the normal chart onto an open neighbourhood
$\underline{U}$ of $o=eP$ in $G/P$.

\begin{proposition}\label{prop:comparison}
Let $\Bbb V$ be an irreducible representation of $G$ with
$P$--irreducible quotient $\Bbb H_0$. Let $(p:\Cal G\to M,\om)$ be a
parabolic geometry of type $(G,P)$ and for a point $u_0\in\Cal G$
consider the normal comparison map $\ps:U\to\underline{U}$.

If $\al\in\Ga(\Cal H_0M)$ is a normal solution of the first BGG
operator determined by $\Bbb V$, then the coordinate functions of
$\al$ with respect to a normal frame of $\Cal H_0M$ are the pullbacks
along $\ps$ of the coordinate functions of a unique corresponding
solution of the first BGG operator determined by $\Bbb V$ in the
corresponding normal frame of $\Cal H_0(G/P)$.
\end{proposition}
\begin{proof}
It clearly suffices to prove this for one choice of normal frames, so
we may use normal frames as in Lemma \ref{lem:adapted}. Thus the
normal frames for the $\Cal H_0$--bundles are just the projections of
some elements of normal frames of corresponding the tractor
bundles. By definition, there is a parallel section $s\in\Ga(VM)$ such
that $\al=\Pi(s)$. From Lemma \ref{lem:basic} we know that the
function $f:U\to\Bbb V$ corresponding to $s$ is given by
$f(\ph(X))=\exp(-X)\cdot f(x_0)$.  Now on the homogeneous model, any
tractor bundle is canonically trivial and the tractor connection is
the flat connection induced by this trivialisation. (Thus on the
homogeneous model all first BGG solutions are normal.)  In particular,
there is a unique parallel section $\underline{s}\in\Ga(\Cal V(G/P))$
such that the corresponding function $\underline{U}\to\Bbb V$ maps $o$
to $f(x_0)$. Again by Lemma \ref{lem:basic}, this function must map
$\underline{\ph}(X)$ to $\exp(-X)\cdot f(x_0)$. This means that the
component functions of $s$ and $\underline{s}$ in any compatible pair
of normal frames are intertwined by $\ps$, and since $\underline{s}$
projects to a solution of the first BGG operator determined by $\Bbb
V$, and conversely this solution determines $\underline{s}$, the
result follows.
\end{proof}

Of course, this result implies that the comparison map
$\ps:U\to\underline{U}$ restricts to a bijection between the zero sets
of the two normal solutions in question. In particular, the zero sets
of normal solutions on curved geometries locally look like zero sets
of solutions on the model. The latter can often be nicely analysed
using algebraic geometry (or even linear algebra),
c.f.~\cite{CGH-projective}; Proposition \ref{prop:comparison} shows
that the {\em same} polynomial systems apply in the curved setting.

To get stronger consequences from this line of argument, one has to
analyse the properties of the comparison maps for individual
structures in more detail. Of course, the comparison map cannot be a
morphism of parabolic geometries unless $(p:\Cal G\to M,\om)$ is
locally flat. (It is a local isomorphism in the locally flat case.)
However, by construction it is compatible with the projections of flow
lines of constant vector fields through $x_0$ respectively through
$o$. Each of these flow lines is a distinguished curve of the geometry
in the sense of Section 5.3 of \cite{cap-slovak-book}, so one gets compatibility
with some canonical curves through $x_0$.  However, not all canonical
curves through $x_0$ are obtained in that way (since the point $u_0$
remains fixed). For most geometries, at least some of the canonical
curves (e.g.~null geodesics in pseudo--Riemannian conformal structures
and chains in CR structures) are uniquely determined by their initial
direction up to parametrisation so one can get nice information on
those. This works particularly well in the situation of projective
structures as discussed in \cite{CGH-projective}. There all canonical curves
are uniquely determined by their initial direction up to
parametrisation, so one simply gets compatibility of $\ps$ with the
unparametrised canonical curves through $x_0$, which is heavily used
in \cite{CGH-projective}.

\subsection{Remark on $G$--types and $P$--types}\label{types}
We only touch on this topic briefly, because it is discussed in detail
in our article \cite{CGH-holonomy} in the more general context of holonomy
reductions of parabolic geometries and in \cite{CGH-projective} for projective
structures.  Consider a representation $\Bbb V$ of $G$, the
corresponding tractor bundle $\Cal VM$ for a parabolic geometry
$(p:\Cal G\to M,\om)$ of type $(G,P)$ and a section $s\in \Ga(\Cal
VM)$. Then for a given point $x\in M$, the image of the fibre $\Cal
G_x$ under the equivariant function corresponding to $s$ is a
$P$--orbit in $\Bbb V$ that we term the \textit{$P$--type of $s$ at
  $x$}. For a given normal solution $\alpha$, it is obvious from
Proposition \ref{prop:firstbgg} that any point in the zero locus of
$\alpha$ lies in a different $P$-type to a point where $\alpha$ is
non-zero. In fact much more information is available and in both the
articles mentioned examples related to normal solutions are
discussed.

It is shown in  \cite{CGH-projective,CGH-holonomy} that in the case of a parallel
section $s$ and a connected base $M$, the $P$--types of all points are
contained in a single $G$--orbit $\Cal O\subset\Bbb V$, which is
called the $G$--type of $s$. If $\Cal O=\cup_i\Cal O_i$ is the
decomposition of the $G$--orbit into $P$--orbits, there is an induced
decomposition $M=\cup_i M_i$ according to $P$--type. In
\cite{CGH-holonomy} it is proved that each of the $M_i$ is an initial
submanifold in $M$. For the homogeneous model, this decomposition is
the decomposition of $G/P$ into orbits under the action of the
isotropy group $H$ of some chosen element of $\Cal O$. Thus, the
general decomposition is called a \textit{curved orbit
  decomposition}. If $\al$ is the normal solution of the first BGG
operator determined by $s$, then the zero--set of $\al$ is stratified
into a union of $P$--types.

Finally, the $H$--orbits in $G/P$ are all homogeneous spaces and thus
homogeneous models of Cartan geometries. In the curved case, one
obtains Cartan geometries of the same types over the individual curved
orbits. The curvatures of these induced Cartan geometries are related
in a precise way (that we describe) to the curvature of the original
geometry of type $(G,P)$.

\section{Examples and applications}\label{appl}
In this section, we show how to apply the general principles and
results obtained above to specific structures. We mainly exploit the
fact that the representations of a classical Lie group $G$ (and hence
the corresponding tractor bundles on parabolic geometries of type
$(G,P)$ for any parabolic subgroup $P$ of $G$) can be built from the
standard representation (and spin representations in the orthogonal
case). This can be used to explicitly describe potential parallel
sections of tractor bundles as well as their projections to the
canonical quotients. The details of this of course strongly depend on
the choices of $G$ and $P$, but the method is quite universal.

\subsection{Normal frames and generalised homogeneous coordinates 
for oriented projective structures}\label{projective}

Here we put $G=SL(n+1,\Bbb R)$, identified with its defining
representation, and $P\subset G$ the stabiliser of the ray spanned by
the first vector in the standard basis, which corresponds to oriented
projective structures. Let us start by showing how to recover the
generalised homogeneous coordinates from \cite{CGH-projective} in our
setting.

The Lie algebra $\frak g=\frak{sl}(n+1,\Bbb R)$ consists of matrices
of the form $\left(\begin{smallmatrix} -\tr(A) & Z \\ X & A
\end{smallmatrix}\right)$ with $X\in\Bbb R^n$, $Z\in\Bbb R^{n*}$ and
$A\in\frak{gl}(n,\Bbb R)$, and $X$, $A$, and $Z$ represent the grading
components $\fg_{-1}$, $\fg_0$, and $\fg_1$, respectively. We start
the construction of normal frames with the \textit{standard tractor
  bundle} $\Cal TM$, which corresponds to the standard representation
$\Bbb R^{n+1}$ of $G$. Let us denote the standard basis of $\Bbb
R^{n+1}$ by $e_0,\dots,e_n$, so $P$ stabilises the line $\Bbb R\cdot
e_0$, and the quotient $\Bbb R^{n+1}/\Bbb R\cdot e_0$ is an
irreducible representation of $P$. In the standard notation for
projective structures the bundle induced by $\Bbb R\cdot e_0$ is a
density bundle usually denoted by $\Cal E(-1)$, while $\Bbb
R^{n+1}/\Bbb R\cdot e_0$ induces the bundle $TM\otimes\Cal E(-1)$. We
will denote the latter bundle by $TM(-1)$ and in general use the
convention that adding ``$(w)$'' to the name of a bundle indicates a
tensor product with the density bundle $\Cal E(w)$.

Given a projective structure, the corresponding parabolic geometry
$(p:\Cal G\to M,\om)$ of type $(G,P)$, and a normal chart $U$, we
consider the corresponding normal frames. The normal frame for the
density bundle $\Cal E(-1)$ is a specific non--vanishing section
determined by the basis vector $e_0$. It will be better, however, to
take the corresponding non--vanishing section $X^0$ of $\Cal E(1)=\Cal
E(-1)^*$ as the basic object, so then $e_0$ corresponds to the density
$(X^0)^{-1}$. The basis vectors $\{e_1,\dots,e_n\}$ then can be used
to complete this to the adapted normal frame of the standard tractor
bundle $\Cal TM$, and projecting them to the quotient, one obtains the
normal frame for $TM(-1)$, which we denote by $\{\xi_1,\dots,\xi_n\}$.

Let us use the obvious identification of $\fg_{-1}$ with $\Bbb R^n$ to
define coordinates and denote the corresponding normal coordinates on
$M$ by $x_1,\dots,x_n\in C^\infty(M,\Bbb R)$. Then for $i=1,\dots,n$
we define $X^i\in\Ga(\Cal E(1))$ by $X^i:=x_iX^0$.

\begin{proposition}\label{prop:proj1}
  The densities $X^0,\dots,X^n\in\Ga(\Cal E(1))$ are exactly the
  (local) generalised homogeneous coordinates on $M$ introduced in
  \cite{CGH-projective}.
\end{proposition}
\begin{proof}
  By its  construction, it is clear that the comparison map to the
  homogeneous model is compatible with the normal coordinates $x_i$ and
  thus also with the densities $X^i$. Comparing with the construction in \cite{CGH-projective}, it thus suffices to show that
  on the homogeneous model the $X^i$ are obtained from the standard
  coordinates on $\Bbb R^{n+1}\setminus\{0\}$. The normal chart in this
  case simply is given by $X\mapsto \exp(X)(e_0)$, and from the
  presentation of $\fg$ above it is obvious that this is given by
  $e_0+\sum_{i=1}^n x_ie_i$, where the $x_i$ are the components of
  $X$. This immediately implies the claim.
\end{proof}

\subsection{Normal solutions for projective structures}\label{proj2}

To proceed further, we need some information on the first BGG
operators on projective structures, which is taken from \cite{BCEG}.
An irreducible representation $\Bbb V$ of $G$ and hence the
corresponding tractor bundle can be determined by its irreducible
quotient $\Bbb H_0$ (viewed as a representation of the semisimple part
of $G_0$) and a single non--negative integer $k$. In this case, we
will say that $\Bbb V$ corresponds to $(\Bbb H_0,k)$. The first BGG
operator is then defined on the bundle $\Cal H_0M(w)$ (induced by
$\Bbb H_0$), for an appropriate value of $w$, and has order $k$. The
range of the operator lies in the bundle induced by the
$P$--representation $S^k\Bbb R^{n*}\odot\Bbb H_0(w')$ for appropriate
$w'$. Here $\odot$ denotes the \textit{Cartan product}, i.e.~the
irreducible component of maximal highest weight in $S^k\Bbb
R^{n*}\otimes\Bbb H_0$. Moreover, if $\Bbb V$ is the representation of
$G$ corresponding to $(\Bbb H_0,1)$, then the $G$--representation
corresponding to $(\Bbb H_0,k)$ is $S^{k-1}\Bbb R^{(n+1)*}\odot\Bbb V$
for any $k\geq 2$.

The basic building blocks of representations of $G$ are the
fundamental representations $\La^r\Bbb R^{n+1}$ for $r=1,\dots,n$ (of
course $\La^n\Bbb R^{n+1}\cong \Bbb R^{(n+1)*}$). The description of
the composition series of $\La^r\Bbb R^{n+1}$ follows readily from the
one of $\Bbb R^{n+1}$. There is a $P$--invariant subspace isomorphic
to $\La^{r-1}\Bbb R^n$ (spanned by the wedge products of basis
elements which involve $e_0$) and the quotient by this is the
irreducible representation $\La^r\Bbb R^n$. In the above description
of irreducible representations of $G$, $\La^r\Bbb R^{n+1}$ corresponds
to $(\La^r\Bbb R^n,1)$ for $r<n$ and $(\Bbb R,2)$ for $r=n$. Hence the
corresponding first BGG operator has order one for $r<n$ and order two
for $r=n$. On the level of associated bundles, the invariant subspace
corresponds to $\La^{r-1}TM(-r)$, while the quotient corresponds to
$\La^rTM(-r)$ and $\La^nTM(-n)\cong\Cal E(1)$.

\begin{theorem}\label{thm:proj2}
  Consider a projective structure on a smooth manifold $M$ of
  dimension $n$, consider a normal chart $U$ on $M$, let
  $x_1,\dots,x_n$ be the corresponding normal coordinates, and let
  $\{\xi_1,\dots,\xi_n\}$ be the normal frame of $TM(-1)$ determined
  by $U$. Then we have:
  
  (1) For $r<n$, any normal solution of the first BGG operator defined
  on $\La^rTM(-r)$ restricts on $U$ to a linear combination with
  constant coefficients of the following sections
$$
\begin{cases}
  \xi_{i_1}\wedge\dots\wedge\xi_{i_r} \quad\text{~for~}1\leq
  i_1<i_2<\dots<i_r\leq n\\
  \sum_jx_j\xi_j\wedge\xi_{i_1}\wedge\dots\wedge\xi_{i_{r-1}}
  \quad\text{~for~}1\leq i_1<i_2<\dots<i_{r-1}\leq n
\end{cases}
$$ 
In the case of the homogeneous model $S^n$, $\xi^i$ is the coordinate
frame associated to a choice of basis of $\fg_-$, each of the sections
listed above is the restriction of a solution, and they form a basis
for the space of solutions of the first BGG operator on $U$.
  
(2) For $k\geq 2$, any normal solution of the ($k$th order) first BGG
operator defined on $\Cal E(k-1)$ is a homogeneous polynomial of degree
$k-1$ in the generalised homogeneous coordinates $\{X^0,\dots,X^n\}$
determined by $U$. In the case of $S^n$, all such polynomials arise as
solutions.

(3) For $k\geq 2$, any normal solution of the ($k$th order) first BGG
operator acting on $\La^rTM(k-r-1)$ can be written as a linear
combination of the sections listed in (1) with coefficients which are
homogeneous polynomials of degree $k-1$ in the generalised homogeneous
coordinates $\{X^0,\dots,X^n\}$.
\end{theorem}
\begin{proof}
  (1) For a matrix $X$ contained in the subspace $\fg_{-1}$ of
  matrices from \ref{projective}, one simply has $\exp(-X)=\Bbb I-X$,
  where $\Bbb I$ denotes the unit matrix. The action on the standard
  representation $\Bbb R^{n+1}$ is thus given by
  $\exp(-X)(e_0)=e_0-\sum_{j=1}^nx_je_j$ and $\exp(-X)(e_i)=e_i$ for
  $i>0$, where the $x_j$ are the components of $X$. Denoting by
  $\{s_0,\dots,s_n\}$ the normal frame determined by $U$ and the basis
  $\{e_0,\dots,e_n\}$ of $\Bbb R^{n+1}$, Lemma \ref{lem:basic} shows
  that any parallel section of the standard tractor bundle $\Cal TM$
  must be given by a linear combination with constant coefficients of
  the sections $s_0-\sum_{j=1}^nx_js_j$ and $s_1,\dots,s_n$. The
  projection to the quotient bundle $TM(-1)$ maps $s_0$ to zero and
  $s_i$ to $\xi_i$ for $i=1,\dots,n$ (compare with Lemma
  \ref{lem:adapted}). Hence any normal solution of the first BGG
  operator on this quotient bundle must be a linear combination with
  constant coefficients of $\xi_0:=\sum_{j=1}^nx_j\xi_j$ and
  $\xi_1,\dots,\xi_n$, which is our claim for $r=1$.

  Now the wedge products of the elements of the standard basis of
  $\Bbb R^{n+1}$ form a basis for $\La^r\Bbb R^{n+1}$, and we conclude
  that any parallel section of $\La^r\Cal TM$ must then be a linear
  combination with constant coefficients of the wedge products of the
  sections $s_0,\dots,s_n$. But the projection of a wedge product of
  sections is just the wedge product of the projections of the
  individual sections. Expanding this for the wedge products involving
  $\xi_0$, we obtain the claimed list of sections. In the case of the
  homogeneous model, each of the sections $s_i$ and thus any wedge
  product of these sections actually is parallel, which together with
  Proposition \ref{prop:homog-frame} implies the last part of (1).

(2) Let us look at the result of (1) in the case $r=n$. Then we get
the sections $\xi_1\wedge\dots\wedge\xi_n$ and (up to a sign that we
may ignore) $x_i\xi_1\wedge\dots\wedge\xi_n$ for $i=1,\dots,n$. The
quotient bundle $\La^nTM(-n)$ is a line bundle which is trivialised
over $U$ by the section $\xi_1\wedge\dots\wedge\xi_n$. Compatibility of
normal sections with constructions on the inducing representations
shows that this section has to be a normal frame for this
bundle. Since $\La^n\Bbb R^{n+1}\cong\Bbb R^{(n+1)*}$, we must have
$\La^nTM(-n)\cong\Cal E(1)$ with the isomorphism mapping
$\xi_1\wedge\dots\wedge\xi_n$ to $X^0$. Hence we conclude that any
normal solution of the first BGG operator on $\Cal E(1)$ is a linear
combination with constant coefficients of the generalised homogeneous
coordinates, which is our claim for $k=2$. 

What we have actually done here was construct a frame for the
cotractor bundle $\Cal T^*M$, over $U$, such that any parallel section
of $\Cal T^*M$ must be a linear combination with constant coefficients
of the frame elements and such that the projections of the frame
elements to the quotient bundle $\Cal E(1)$ are exactly the
generalised homogeneous coordinates $X^i$. Now we immediately conclude
that the symmetric power $S^{k-1}\Cal T^*M$ has $\Cal E(k-1)\cong
S^{k-1}\Cal E(1)$ as its natural quotient. The symmetric products of
$k-1$ elements of our frame for $\Cal T^*M$ form a frame for $S^{k-1}\Cal
T^*M$ such that any parallel section must be a linear combination of
the frame elements with constant coefficients. Projecting a symmetric
product of sections to $\Cal E(k-1)$, one of course obtains the
product of the projections of the individual sections. This completes
the proof of (2).

(3) This now immediately follows since from (1) and (2) we can form a
frame of $S^{k-1}\Cal T^*M\otimes\La^r\Cal TM$ such that any parallel
section of this bundle must be a linear combination with constant
coefficients of the frame elements. As we have noticed above, the
first BGG operator defined on $\La^rTM(k-r-1)$ comes from the
subbundle corresponding to $S^{k-1}\Bbb R^{(n+1)*}\odot \La^r\Bbb
R^{n+1}$. Of course, we can construct a frame of this subbundle which
consists of linear combinations of tensor products of elements of the
frames of the two factors as constructed in (1) and (2). The
projection of such a tensor product is again the tensor product of the
projections of the factors, which implies the result.
\end{proof}

\subsection{A more involved example for projective
    structures}\label{proj3} 

  The result in part (3) of Theorem \ref{thm:proj2} is not optimal,
  since only certain linear combinations of the elements of the frame
  of $S^{k-1}\Cal T^*M\otimes\La^r\Cal TM$ constructed in the proof
  will actually lie in $S^{k-1}\Cal T^*M\odot\La^r\Cal TM$.  Going
  into more details on the decomposition of tensor products, one can
  improve the result. Using similar considerations, one can extend the
  results of part (1) to operators defined on more complicated
  bundles, and we discuss an example of this. We want to describe
  solutions of the first BGG operators on the bundle $S^2T^*M(4)$,
  which is of order one. To formulate the result, it will be better to
  first recast the result of Theorem \ref{thm:proj2} for $r=n-1$ in
  terms of the bundle $T^*M(2)$. We have chosen this example since
  both these results are of interest for Riemannian geometry, see
  \ref{Riemann} below.

  We can identify $\La^{n-1}\Bbb R^{n+1}$ with $\La^2\Bbb
  R^{(n+1)*}$. It then follows immediately that for the bundle
  $\La^2\Cal T^*M$, the irreducible quotient is $T^*M(2)$ while the
  subbundle contained in there is $\La^2T^*M(2)$. Now consider the
  normal frame $\{\ph_1,\dots,\ph_n\}$ for $T^*M(2)$. Under the
  isomorphism $\La^{n-1}TM(n-1)\cong T^*M(2)$ the element $\ph_i$
  corresponds of course (up to a sign which is not relevant for us) to
  the wedge product of the $\xi_j$ for $j\neq i$. Using this, we can
  read off the following from part (1) of Theorem \ref{thm:proj2}:

  \begin{corollary}\label{cor:proj3}
    In terms of a normal frame $\{\ph_1,\dots,\ph_n\}$ for $T^*M(2)$
    and the corresponding normal coordinates $x_1,\dots,x_n$, any
    normal solution of the first BGG operator on $T^*M(2)$ can be
    written as a linear combination with constant coefficients of the
    forms $\ph_i$ for $i=1,\dots,n$ and $x_i\ph_j-x_j\ph_i$ for
    $i<j$. On the homogeneous model the $\ph_i$ are coordinate forms,
    each of the listed forms is a solution, and they form a basis for
    the space of all solutions.
\end{corollary}

We can use an analogue of homogenising (as in projective algebraic geometry) to
present the result in a more uniform way.  Let us start with a normal
frame $\{\tilde\ph_1,\dots,\tilde\ph_n\}$ for the bundle $T^*M(1)$. Of
course, we simply get $\ph_i=X^0\tilde\ph_i$ for $i=1,\dots,n$, where
$X^0\in\Ga(\Cal E(1))$ is the non--vanishing section defining the normal
frame, see \ref{projective}. Since the product $x_iX^0$ is just the
homogeneous coordinate $X^i$, we can write the remaining sections from
the corollary as $X^i\tilde\ph_j-X^j\tilde\ph_i$ for $1\leq i<j\leq
n$. Even more uniformly, we can put $\tilde\ph_0=0$, and then obtain
all sections as $X^i\tilde\ph_j-X^j\tilde\ph_i$ for $0\leq i<j\leq n$.
We can now use this result to give a complete list of potential normal
solutions of the first BGG operator defined on $S^2T^*M(4)$.

\begin{proposition}\label{prop:proj3}
  Consider a normal frame $\{\ph_{ij}:1\leq i\leq j\leq n\}$ for the
  bundle $S^2T^*M(4)$ on a manifold endowed with a projective
  structure. Then any normal solution of the first BGG operator
  defined on this bundle can be written as a linear combination with
  constant coefficients of the following sections:
  \begin{alignat*}{2}
    &\ph_{ij} && i\leq j\\
    &x_k\ph_{ij}-x_j\ph_{ik} && i\leq j<k\\
    &x_k\ph_{ij}-x_i\ph_{jk} && i<j\leq k\\
    &x_jx_k\ph_{ii}-x_ix_k\ph_{ij}-x_ix_j\ph_{ik}+{x_i}^2\ph_{jk}
    && i<j\leq k\\
    &x_kx_\ell\ph_{ij}-x_jx_k\ph_{i\ell}-x_ix_\ell\ph_{jk}+x_ix_j\ph_{k\ell} 
    &\quad& i<j<k\leq\ell\\
    & x_jx_\ell\ph_{ik}-x_jx_k\ph_{i\ell}-x_ix_\ell\ph_{jk}+x_i
    x_k\ph_{j\ell} && i<j\leq k<\ell 
  \end{alignat*}
On the homogeneous model, each of these sections is a solution, and
they form a basis for the space of all solutions. 
\end{proposition}
\begin{proof}
  We have seen above that the $P$--irreducible quotient of the
  $G$--irreducible representation $\La^2\Bbb R^{(n+1)*}$ induces the
  bundle $T^*M(2)$. Consequently, the $P$--irreducible quotient of
  $S^2(\La^2\Bbb R^{(n+1)*})$ induces $S^2T^*M(4)$. Elementary
  representation theory shows that $S^2(\La^2\Bbb R^{(n+1)*})$ splits
  into two irreducible components. The smaller of those is mapped
  isomorphically to $\La^4\Bbb R^{(n+1)*}$ by the wedge product of
  forms, which easily implies that it is contained in the kernel of
  the projection to the $P$--irreducible quotient. Hence the right
  tractor bundle to start with is induced by the Cartan product, which
  is the kernel of the wedge product.
  
  To obtain Corollary \ref{cor:proj3} we have used a local frame for
  $\La^2\Cal T^*M$, whose elements we number as $\ps_{ab}$ with $0\leq
  a<b\leq n$. In the notation of that corollary, the projections of
  the elements of the frame are given by $\Pi(\ps_{0i})=\ph_i$ for
  $i=1,\dots,n$ and $\Pi(\ps_{ij})=(x_i\ph_j-x_j\ph_i)$ for $1\leq
  i<j\leq n$. On the homogeneous model, the $\ps_{ab}$ form a basis
  for the space of parallel sections, while in general any parallel
  section is a linear combination with constant coefficient of these
  sections.
  
  If we order pairs of indices in some way, the symmetric products
  $\ps_{ab}\vee\ps_{cd}$ with $ab\leq cd$ form a frame for
  $S^2(\La^2\Cal T^*M)$. Now we can modify this frame in such a way
  that some of its elements lie in the kernel of the wedge product
  while the remaining ones project isomorphically to a frame for
  $\La^4\Cal T^*M$. Ignoring the latter elements, we arrive at a frame
  for our tractor bundle, which is a basis for the space of parallel
  sections on the homogeneous model, while in general any parallel
  section is a linear combinations with constant coefficients of
  elements of the frame.

Working this out, we see that the frame in question arises from the
following elements 
\begin{alignat*}{2}
  &\ps_{ab}\vee\ps_{ac} && 0\leq a<b\leq c\leq n\\
  &\ps_{ab}\vee\ps_{cd}+\ps_{ac}\vee\ps_{bd} &\quad& 0\leq a<b\leq c<d\leq
  n \\
  &\ps_{ad}\vee\ps_{bc}+\ps_{ac}\vee\ps_{bd} && 0\leq a<b<c\leq d\leq n
\end{alignat*}
Now the projections of a symmetric product are just the symmetric
products of the projections of the individual factors, and after recombining some elements we arrive at the claimed list. 
\end{proof}

\subsection{Applications to (pseudo--)Riemannian
  geometry}\label{Riemann} 

We will always use the term ``Riemannian'' to also include metrics of
indefinite signature. Indeed, the results we discuss here are
independent of the signature in question.  It is a standard technique
in Riemannian geometry to study conformal changes of metrics or, more
technically speaking, to study the conformal structure induced by a
Riemannian metric.  Properties of Riemannian manifolds depending only
on the induced conformal class are then considered as particularly
robust. However, via its Levi--Civita connection, a Riemannian metric
on a smooth manifold also determines a projective structure.  The
distinguished paths of this structure are the geodesic paths of the
metric, so this point of view in particular captures aspects related
to geodesics. Note that a Riemannian metric gives rise to a canonical
volume density and thus to a trivialisation of all projective density
bundles. 

We believe that the projective structure induced by a Riemannian
metric has by far not been used up to its potential. This
approach seems particularly promising since several of the fundamental
differential equations of Riemannian geometry admit a projectively
invariant interpretation. This was pointed out and studied in
\cite{eastwood-notes}. In particular, this is true for the
infinitesimal automorphism equation, which is among the examples we
study below. Here we want to continue exploring this point of view. In
particular, we want to point out that normality as a solution of a
projective first BGG operator gives rise to highly interesting
conditions on solutions of some important natural differential
equations in Riemannian geometry.

\medskip

Returning to projective structures, we can explicitly interpret the
first BGG operators of order one on symmetric powers of the cotangent
bundle. From the description in \ref{proj2}, is is clear that these
must map sections of $S^kT^*M(w)$ to sections of $S^{k+1}T^*M(w')$ for
appropriate weights $w$ and $w'$. It then follows easily that in terms
of a chosen connection in the projective class, the operator is just
given by symmetrising the covariant derivative (and the weight is
chosen in such a way that this does not depend on the choice of the
connection in the class).

If we look at the projective structure induced by a Riemannian metric,
then as noted above the weights do not play a role, so one obtains the
standard Killing operators on symmetric tensors. In particular, for
$k=1$ its solutions (which are one--forms) can be interpreted as
vector fields using the metric, and then they are exactly the
infinitesimal isometries. For $k=2$, one obtains (via the metric)
symmetric Killing tensors of valence two which are important in
several parts of Riemannian geometry and integrable systems. The
solutions which are normal in the projective sense form an interesting
subclass of Killing vectors respectively Killing tensors, which, to
our knowledge, has not been studied in any detail so far.

We start with the case $k=1$, phrase things in terms of vector fields
rather than one--forms, and use the Penrose abstract index notation.
\begin{proposition}
  Let $(M,g)$ be a Pseudo--Riemannian manifold of dimension $n\geq 2$
  with Levi--Civita connection $\nabla$. Then a vector field
  $\xi=\xi^k$ on $M$ is a Killing field if and only if the associated
  one--form $\ps_i:=g_{ij}\xi^j$ is a solution of the projective first
  BGG operator defined on $T^*M(2)$. 
  
  If $n=2$ then any such solution is normal in the projective
  sense. If $n\geq 3$, then let $R_{ij}{}^k{}_\ell$ be the Riemann
  curvature tensor of $g$, $R_{ij}:=R_{ki}{}^k{}_j$ its Ricci
  curvature. A solution $\ps_i$ is normal in the projective sense if
  and only if
  \begin{equation}
    \label{eq:Knorm}
    R_{ij}{}^k{}_\ell\ps_k=\tfrac{2}{n-1}\ps_{[i}R_{j]\ell}. 
  \end{equation}
  In particular, any Killing vector field $\xi$ which is normal in the
  projective sense can be written as a linear combination with
  constant coefficients of the vector fields obtained via the metric
  from the forms listed in Corollary \ref{cor:proj3}.
\end{proposition}
\begin{proof}
  In this proof we use significantly more information on BGG sequences
  and the normalisation condition for parabolic geometries than
  outlined in section \ref{2}. The necessary facts can be found in
  \cite{cap-slovak-book} and \cite{mrh-thesis}.

  We have noted above that the first BGG operator on $T^*M(2)$ is
  given by taking one covariant derivative and then symmetrising the
  two indices. It is well known that applying this to the one--form
  obtained from a vector field via the metric one obtains the Killing
  equation.
  
  We have noted in \ref{proj3} that the tractor bundle inducing this
  operator is $\La^2\Cal T^*M$, the second exterior power of the dual
  of the standard tractor bundle. We have also seen there that the
  composition series of this tractor bundle consists of just two parts,
  a subbundle isomorphic to $\La^2T^*M(2)$ and the irreducible
  quotient $T^*M(2)$. We write $\Pi$ for the projection map to this
  irreducible quotient. Likewise, each of the bundles
  $\La^kT^*M\otimes \La^2\Cal T^*M$ of $k$--forms with values in our
  tractor bundle has a composition series with two factors, a
  subbundle isomorphic to $\La^kT^*M\otimes\La^2T^*M(2)$, the quotient by
  which is isomorphic to $\La^kT^*M\otimes T^*M(2)$.
  
  In the construction of the BGG operators, one uses the so--called
  splitting operator $S:\Ga(T^*M(2))\to\Ga(\La^2\Cal T^*M)$ and the
  bundle maps 
$$\partial^*:\La^kT^*M\otimes \La^2\Cal T^*M\to
  \La^{k-1}T^*M\otimes \La^2\Cal T^*M .$$ 
The curvature $\ka$ of the
  canonical Cartan connection can be interpreted as a two--form on $M$
  with values in the bundle $\frak{sl}(\Cal TM)$. In particular, the
  values of $\ka$ act on any tractor bundle. Given a section $\ps$ of
  the bundle $T^*M(2)$, we can first apply the splitting operator and
  then act with the curvature on the result to obtain a two--form with
  values in $\La^2\Cal T^*M$, which is usually denoted by $\ka\bullet
  S(\ps)$. Now it follows from the general theory (see
  \cite{BCEG,mrh-coupling,hsss}) that normality of the solution $\ps$
  is equivalent to the fact that $\partial^*(\ka\bullet S(\ps))=0$.
  
  Since $\partial^*$ preserves homogeneity, it vanishes on the
  subbundle 
$$
\La^2T^*M\otimes\La^2T^*M(2)\subset
  \La^kT^*M\otimes\La^2\Cal T^*M
$$ 
and thus factors through the
  quotient $\La^2T^*M\otimes T^*M(2)$ and its values lie in the
  subbundle $T^*M\otimes\La^2T^*M(2)\subset T^*M\otimes\La^2\Cal
  T^*M$. Hence $\partial^*$ acts between two isomorphic completely
  reducible bundles, which both split into two non--isomorphic
  irreducible summands, so it has to act by a multiple of the identity
  on each of the two summands. It also follows from general results
  that $\partial^*$ must map onto the subbundle in question so both
  these multiples must be non--zero.
  
  From the well known form of $\ka$ (see e.g.~section 5.1.1 of
  \cite{mrh-thesis}) and the fact that $\Pi(S(\ps))=\ps$, we conclude
  that the projection of $\ka\bullet S(\ps)$ to the quotient
  $\La^2T^*M\otimes T^*M(2)$ must be a non--zero multiple of
  $C_{ij}{}^k{}_\ell\ps_k$, where $C_{ij}{}^k{}_\ell$ denotes the
  projective Weyl--curvature. If $n=2$ then it is well known that the
  projective Weyl--curvature vanishes identically, so normality
  follows. For $n\geq 3$, the projective Weyl curvature satisfies the
  first Bianchi--identity, so $ C_{ij}{}^k{}_\ell\ps_k$ lies in the
  kernel of the alternation over the three lower indices. This
  condition characterises exactly one of the two irreducible summands
  mentioned above, so we conclude that normality of $\ps$ is
  equivalent to $C_{ij}{}^k{}_\ell\ps_k=0$.
  
  Now the projective Weyl tensor is given as the tracefree part (with
  respect to the Ricci type contraction) of the curvature of any
  connection in the projective class. In the case of the projective
  structure determined by a Riemannian metric, we can use the Riemann
  curvature tensor. The Ricci type contraction of this is just the
  classical Ricci tensor and hence is symmetric. Then the well known
  formula for the projective Weyl tensor shows that
$$
C_{ij}{}^k{}_\ell=R_{ij}{}^k{}_\ell-\frac{2}{n-1}\delta^k_{[i}R_{j]\ell}. 
$$
Hooking $\ps_k=\xi^ag_{ak}$ into this expression we immediately get
\eqref{eq:Knorm}.
\end{proof}

Finally we note that we may also interpret the result of Proposition
\ref{prop:proj3} in a Riemannian context: Given a section
$\xi=\xi^{ab}$ of the bundle $S^2TM$ on a Riemannian manifold, one can
use the metric to lower the two indices and then consider the equation
$\nabla_{(a}\xi_{bc)}=0$. Solutions of this equation are called
Killing--2--tensors and they play a crucial role in the study of
symmetries of mechanical systems. So similarly as discussed for
Killing fields above, solutions which satisfy the condition of
projective normality can be obtained from linear combinations with
constant coefficients of the sections listed in Proposition
\ref{prop:proj3}. Thus the condition of projective normality should be
very interesting for Riemannian geometry. Using results from section
5.3 of \cite{mrh-thesis}, this condition can be made explicit along
similar lines as the one for Killing forms. Since the result is rather
involved, we do not write it out explicitly here.

\subsection{Example 2: Conformal structures}\label{conformal}
We next apply our results to (oriented) pseudo--Riemannian conformal
structures of arbitrary signature. As we shall see, this looks a bit
more complicated than the projective case, but making our results
explicit is still straightforward. In the conformal case, the first
BGG operators coming from tractor bundles induced by fundamental
representations, are the operator governing Einstein rescalings and
the conformal Killing equations on differential forms. So these are of
strong interest in conformal geometry and are studied intensively.

By classical results going back to E.~Cartan, oriented conformal
structures of signature $(p,q)$ in dimension $n=p+q\geq 3$ are
equivalent to parabolic geometries of type $(G,P)$, where
$G=SO(p+1,q+1)$ and $P\subset G$ is the stabiliser of an isotropic
ray in the standard representation $\Bbb R^{p+1,q+1}$ of $G$. To pass
to matrix representations, one usually takes the standard basis of
$\Bbb R^{p+1,q+1}$ numbered as $e_0,\dots,e_{n+1}$ and defines the
inner product by $\langle e_0,e_{n+1}\rangle=\langle
e_{n+1},e_0\rangle=1$, $\langle e_i,e_i\rangle=\eps_i$ for
$i=1,\dots,n$ and all other inner products vanishing. Here $\eps_i=1$
for $i=1,\dots,p$ and $\eps_i=-1$ for $i=p+1,\dots,n$. So
$e_1,\dots,e_n$ form an orthonormal basis for the standard inner
product of signature $(p,q)$ on an $n$--dimensional subspace and the
other two elements are a light cone basis for the orthocomplement of
that subspace. For this choice of basis, one obtains
$$
\frak g=\frak{so}(p+1,q+1)=\left\{\begin{pmatrix} a & Z & 0 \\ 
X & A & -\Bbb I_{p,q}Z^t \\ 0 & -X^t\Bbb I_{p,q} & -a \end{pmatrix}\right\},
$$ 
where $a\in\Bbb R$, $X\in\Bbb R^n$, $Z\in\Bbb R^{n*}$, $A\in\frak{so}
(p,q)$, and $\Bbb I_{p,q}$ is the diagonal matrix with entries
$\eps_i$, see section 1.6.2 of \cite{cap-slovak-book}. The grading corresponding
to the parabolic subgroup $P$ is of the form
$\fg=\fg_{-1}\oplus\fg_0\oplus\fg_1$, with $\fg_{-1}$ spanned by the elements $X$,
$\fg_0$ by $a$ and $A$, and $\fg_1$ by $Z$. In particular, we can
simply use the components of $X$ as coordinates on $\fg_-=\fg_{-1}$.
Then we compute
\begin{equation}\label{conf-exp}
\exp\begin{pmatrix} 0 & 0 & \dots & 0 & 0\\ -x_1 & 0 &\dots & 0 & 0\\
\vdots & \vdots & & \vdots & \vdots \\  -x_n & 0 &\dots & 0 & 0 \\
0 & \eps_1x_1 & \dots & \eps_nx_n & 0 \end{pmatrix}=
\begin{pmatrix} 1 & 0 & \dots & 0 & 0\\ -x_1 & 1 &\dots & 0 & 0\\
\vdots & \vdots & & \vdots & \vdots \\  -x_n & 0 &\dots & 1 & 0 \\
\frac{\sum\eps_ix_i^2}{2} & \eps_1x_1 & \dots & \eps_nx_n & 1\end{pmatrix}
\end{equation}
From this, we can immediately read off the basic result on parallel
sections of the standard tractor bundle, which by definition is
induced by the standard representation $\Bbb R^{p+1,q+1}$.

\begin{proposition}\label{prop:conf1}
  Let $(M,[g])$ a smooth manifold endowed with a conformal
  pseudo--Riemannian structure of signature $(p,q)$ and let $\Cal TM$
  be the standard tractor bundle. For a normal chart $U\subset M$ with
  corresponding normal coordinates $x_1,\dots,x_n$ let
  $\{s_0,\dots,s_{n+1}\}$ be the normal frame of $\Cal TM$
  corresponding to the standard basis of $\Bbb R^{p+1,q+1}$.

  Then any parallel section of $\Cal T|_U$ is a linear combination
  with constant coefficients of the sections $\tilde s_0:=s_0-\sum
  x_is_i+(\frac12\sum\eps_i x_i^2)s_{n+1}$, $\tilde s_i:=s_i+\eps_i x_i s_{n+1}$
  for $i=1,\dots,n$ and $\tilde s_{n+1}:=s_{n+1}$.
\end{proposition}

Viewing $\Bbb V:=\Bbb R^{p+1,q+1}$ as a representation of the
parabolic subgroup $P$, there is an obvious $P$--invariant filtration
of the form $\Bbb V=\Bbb V^{-1}\supset\Bbb V^0\supset\Bbb V^1$.  Here
$\Bbb V^1$ is the isotropic line stabilised by $P$ and $\Bbb V^0$ is
its orthogonal space. Our choice of the quadratic form was made in
such a way that $\Bbb V^1=\Bbb R\cdot e_0$ while $\Bbb V^0$ is spanned
by $e_0,e_1,\dots,e_n$. Now $\Bbb V/\Bbb V^0$ is a (non--trivial)
one--dimensional representation of $P$ (spanned by the image of
$e_{n+1}$ in the quotient). The natural line bundle induced by this
representation is a density bundle, usually denoted by $\Cal
E[1]$. The quotient $\Bbb V^0/\Bbb V^1$ is $n$--dimensional (spanned
by the images of $e_1,\dots,e_n$). The bundle induced by this quotient
turns out to be $T^*M[1]=T^*M\otimes\Cal E[1]$. 

Passing to associated bundles, the filtration of $\Bbb V$ induces a
filtration $\Cal TM=\Cal T^{-1}M\supset\Cal T^0M\supset\Cal T^1M$ by
smooth subbundles such that $\Cal TM/\Cal T^0M\cong\Cal E[1]$ and
$\Cal T^0M/\Cal T^1M\cong T^*M[1]$. From the construction it is clear
that the normal frame $\{s_0,\dots,s_{n+1}\}$ from Proposition
\ref{prop:conf1} has the property that $s_0$ is a normal frame for
$\Cal T^1M$, while $\{s_0,\dots,s_n\}$ form a normal frame for $\Cal
T^0M$. Further, Lemma \ref{lem:adapted} shows that the natural
projection $\Cal TM\to\Cal TM/\Cal T^0M\cong\Cal E[1]$ maps
$s_0,\dots,s_n$ to zero and $s_{n+1}$ to a normal frame of $\Cal
E[1]$. Likewise, the projection $\Cal T^0M\to \Cal T^0M/\Cal T^1M\cong
T^*M[1]$ annihilates $s_0$ and maps $s_1,\dots,s_n$ to a normal frame
of $T^*M[1]$. Using this, we can define the conformal version of
generalised homogeneous coordinates.

\begin{definition}\label{def:conf-homog}
Let $(M,[g])$ be a smooth manifold endowed with a pseudo--Riemannian
conformal structure of signature $(p,q)$. Let $U\subset M$ be a normal
chart and let $X^0$ be the nowhere vanishing section of $\Cal E[1]|_U$ 
defining a normal frame for this bundle. Then we define the
\textit{generalised homogeneous coordinates} for $M$ on $U$ as the
sections $X^0,\dots,X^{n+1}$ of $\Cal E[1]$ defined by $X^i:=\eps_ix_iX^0$
for $i=1,\dots,n$ and $X^{n+1}=\frac{\sum_j\eps_jx_j^2}{2}X^0$. 
\end{definition}
Observe that the generalised homogeneous coordinates satisfy the
relation $2X^0X^{n+1}+\sum_{i=1}^n\eps_i (X^i)^2=0$ (the products can
be interpreted as sections of $\Cal E[2]$). This exactly corresponds
to the defining equation of the light cone in $\Bbb R^{p+1,q+1}$ for
our choice of inner product. This is very natural in view of the fact
that the homogeneous model of our geometry is the space of isotropic
rays in $\Bbb R^{p+1,q+1}$, which is isomorphic to $S^p\times S^q$. 

\medskip

From the filtration of the standard representation one can immediately
read off the filtration for its exterior powers. Looking at $\La^r\Bbb
V$, one gets a $P$--invariant subspace spanned by the wedge products
which involve $e_0$ but not $e_{n+1}$. This is contained in a
$P$--invariant subspace spanned by the wedge products which either
involve $e_0$ or don't involve $e_{n+1}$. The $P$--irreducible
quotient of $\La^r\Bbb V$ is spanned by the images of all elements of
the form $e_{i_1}\wedge\dots \wedge e_{i_{r-1}}\wedge e_{n+1}$ with
$1\leq i_1<\dots<i_{r-1}\leq n$. This shows that the $P$--irreducible
quotient of the tractor bundle $\La^r\Cal T$ is isomorphic to
$\La^{r-1}T^*M[r]$. In particular, for $r=2$, one gets $T^*M[2]\cong
TM$.

The first BGG operators corresponding to the fundamental
representations are all well known and intensively studied in the
literature. The first BGG operator on $\Cal E[1]$ is the second order
operator governing \textit{almost Einstein} scales, see
e.g.~\cite{gover-aes}. An important feature of this operator is that on
any conformal manifold, all its solutions are automatically normal.
For $r\geq 2$, the first BGG operator on $\La^{r-1}T^*M[k]$ coming
from the $r$th exterior power of the standard tractor bundle is the
\textit{conformal Killing operator} on $(r-1)$--forms, see
e.g.~\cite{gover-silhan-forms,semmelmann}. This is the first order operator given by
taking one covariant derivative and then projecting to the highest
weight component. Solutions of this operator are called
\textit{conformal Killing forms} and normal solutions are called
\textit{normal conformal Killing forms}, see \cite{leitner-normal}.

Finally, in parallel to the projective case, there are some first BGG
operators of higher order which are easy to describe. Starting from
the symmetric power $S^k\Cal T$ one obtains the $P$--irreducible
quotient $\Cal E[k]$ and the first BGG operator on this bundle is of
order $k+1$. Likewise, forming the Cartan product $S^{k-1}\Bbb V\odot
\La^r\Bbb V$ one obtains a $G$--representation with $P$--irreducible
quotient inducing the bundle $\La^{r-1}T^*M[r+k-1]$ and the first BGG
operator defined on this bundle is of order $k$. 

\begin{theorem}\label{thm:conf2}
Let $(M,[g])$ be a smooth manifold endowed with a pseudo--Riemannian
conformal structure of signature $(p,q)$. Let $U\subset M$ be a normal
chart and $x_1,\dots,x_n$ the corresponding normal coordinates. 

(1) Any normal solution of the first BGG operator on $\Cal E[k]$ over
$U$ is given as a homogeneous polynomial of degree $k$ in the
generalised homogeneous coordinates $X^0,\dots,X^{n+1}$. In
particular, any almost Einstein scale on $M$ restricts on $U$ to a
linear combination with constant coefficients of the generalised
homogeneous coordinates. On the homogeneous model $S^p\times S^q$ any
homogeneous polynomial in the generalised homogeneous coordinates
defines a solution.

(2) Let $\{\xi_1,\dots,\xi_n\}$ be a normal frame of the tangent
bundle over $U$. Then any normal conformal Killing vector field on $U$
is a linear combination with constant coefficients of the fields
$\xi_1,\dots,\xi_n$, $\sum_{i=1}^n x_i\xi_i$,
$\eps_jx_j\xi_i-\eps_ix_i\xi_j$ for $1\leq i<j\leq n$, and
$\frac{1}{2}(\sum_{\ell=1}^n\eps_\ell
x_\ell^2)\xi_i+\eps_ix_i\sum_{j=1}^nx_j\xi_j$.  On the homogeneous
model, the $\xi_i$ are a coordinate frame and the vector fields listed
above form a basis for the space of conformal Killing fields.

(3) Let $\{\ph_1,\dots,\ph_n\}$ be a normal frame of the bundle $T^*M$
over $U$ and for $r\geq 2$ let us denote by $\ph_{i_1\dots i_r}$ the
section $(X^0)^{r+1}\ph_{i_1}\wedge\dots\wedge\ph_{i_r}$ of the bundle
$\La^rT^*M[r+1]$. Then any normal conformal Killing $r$--form on $M$ is
a linear combination with constant coefficients of the following forms
(with a hat denoting omission)
\begin{align*}
&\ph_{i_1\dots i_r},\quad \textstyle\sum_{j=1}^nx_j\ph_{ji_1\dots i_{r-1}},\quad 
\sum_{j=1}^{r+1}(-1)^{j-1}\eps_{i_j}x_{i_j}\ph_{i_1\dots
  \widehat{i_j}\dots i_{r+1}}\quad\text{and}\\
&\textstyle(-1)^r\frac{1}{2}(\sum_\ell\eps_\ell x_\ell^2)\ph_{i_1\dots i_{i_r}}+\sum_{j=1}^r(-1)^{r-j}\eps_{i_j}x_{i_j}\left(\textstyle\sum_\ell
x_\ell\ph_{\ell i_1\dots\widehat{i_j}\dots i_r}\right),
\end{align*}
where in each case the $(i_1,\dots)$ runs through all strictly
increasing sequences of numbers between $1$ and $n$.

On the homogeneous model, the forms in the list constitute a basis for
the space of all conformal Killing forms of degree $r$.

(4) For $r\geq 1$ and $k\geq 2$, any normal solution of the first BGG
operator on the bundle $\La^rT^*M[r+k]$ (which is of order $k$) can be
written as a linear combination of the sections listed in (2) (for
$r=1$) respectively (3) with coefficients which are homogeneous
polynomials of degree $k-1$ in the generalised homogeneous coordinates
$X^0,\dots,X^{n+1}$ from Definition \ref{def:conf-homog}.
\end{theorem}
\begin{proof}
Since different normal frames for the normal chart $U$ are obtained
from each other via linear combinations with constant coefficients, it
suffices to prove each of the claims for one normal frame of the
bundle in question. We start with the normal frame
$\{s_0,\dots,s_{n+1}\}$ for the standard tractor bundle $\Cal TM$ from
Proposition \ref{prop:conf1}. As we have observed already, the
projection $\Cal TM\to\Cal E[1]$ maps $s_0,\dots,s_n$ to zero and
$s_{n+1}$ to a nowhere vanishing section $X^0$ of $\Cal E[1]$, which
constitutes a normal frame. Now the sections $\tilde s_i$ from
Proposition \ref{prop:conf1} project to the generalised homogeneous
coordinates $X^0,\dots,X^{n+1}$, which implies (1) for $k=1$.

To obtain (1) in the case $k\geq 2$, we just have to observe that
Proposition \ref{prop:conf1} immediately implies that any parallel
section of the bundle $S^k\Cal TM$ must be a linear combination with
constant coefficients of the symmetric products $\tilde
s_{i_1}\vee\dots\vee\tilde s_{i_k}$. Projecting such a product to the
quotient bundle $\Cal E[k]$, one obtains the product of the
projections of the individual factors.

\medskip

(2) The wedge products $s_i\wedge s_j$ with $0\leq i<j\leq n+1$ form a
normal frame for the bundle $\La^2\Cal TM$. The projection to the
irreducible quotient bundle $T^*M[2]\cong TM$ annihilates a wedge
product if either $i=0$ or $j<n+1$ and maps the sections $s_i\wedge
s_{n+1}$ for $i=1,\dots n$ to a normal frame of $TM$ and we use this
frame as $\{\xi_1,\dots,\xi_n\}$. On the other hand, Proposition
\ref{prop:conf1} implies that any parallel section of $\La^2\Cal TM$
must be a linear combination with constant coefficients of the
sections $\tilde s_i\wedge\tilde s_j$ for $0\leq i<j\leq
n+1$. Expanding the formulae for the $\tilde s_\ell$ from Proposition
\ref{prop:conf1} and using the observations on projections above, one
obtains the list in (2).

\medskip

(3) Here we consider the wedge products $s_{i_1}\wedge\dots\wedge
s_{i_{r+1}}$ which form a normal frame for $\La^{r+1}\Cal
TM$. Projecting to the quotient bundle $\La^rT^*M[r+1]$ kills any wedge
product with $i_1=0$ or $i_{r+1}<n+1$. The wedge products
$s_{i_1}\wedge\dots\wedge s_{i_r}\wedge s_{n+1}$ with $1\leq
i_1<\dots<i_r\leq n$ descend to a normal frame of $\La^rT^*M[r+1]$,
which we use as $\ph_{i_1\dots i_r}$. By naturality of normal frames,
this is of the form claimed in (3) for $\ph_i=(X^0)^{-2}\xi_i$ with
the vector fields $\xi_i$ from (2). Again by Proposition
\ref{prop:conf1}, any parallel section of $\La^{r+1}\Cal TM$ is a
linear combination with constant coefficients of the wedge products
$\tilde s_{i_1}\wedge\dots\wedge \tilde s_{i_{r+1}}$ and expanding
this using the formula for the $\tilde s$ from that proposition and
then projecting leads to the list in (3).
 
\medskip

(4) As we have observed above, the ($k$th order) first BGG operator on
$\La^rT^*M[r+k]$ comes from the tractor bundle induced by $S^{k-1}\Bbb
V\odot\La^{r+1}\Bbb V\subset S^{k-1}\Bbb V\otimes \La^{r+1}\Bbb V$. Thus
this part can be proved exactly as the projective counterpart in part
(3) of Theorem \ref{thm:proj2}.
\end{proof}

\begin{remark}\label{rem:cKforms}
  As in the projective case, one can use homogenisation to present the
  results in (2) and (3) in a more uniform way. For the case of
  conformal Killing vectors, we start with a normal frame
  $\{\bar\xi_1,\dots,\bar\xi_n\}$ for the bundle $T^*M[1]\cong
  TM[-1]$, and define $\bar\xi_{n+1}:=-\sum_{i=1}^nx_i\bar\xi_i$. Then
  we can write the sections in (2) as $X^0\bar\xi_i$ for
  $i=1,\dots,n+1$ and $X^j\bar\xi_i-X^i\bar\xi_j$ for $1\leq i< j\leq
  n+1$.

  For the conformal Killing forms in (3), one similarly looks at
  elements $\bar\ph_{i_1\dots i_r}:=(X^0)^r\ph_{i_1}\wedge\dots\wedge
  \ph_{i_r}$ for $1\leq i_1<\dots <i_r\leq n$ of the bundle
  $\La^rT^*M[r]$. For $1\leq i_1<\dots<i_{r-1}\leq n$, one then
  defines $\bar\ph_{i_1\dots i_r n+1}:=\sum_\ell
  x_\ell(X^0)^r\ph_{i_1}\wedge\dots\wedge
  \ph_{i_{r-1}}\wedge\ph_\ell$. Then the forms in (3) can be written
  as $X^0\bar\ph_{i_1\dots i_r}$ and
  $\sum_{j=0}^r(-1)^{j-1}X^{i_j}\bar\ph_{i_0\dots\widehat{i_j}\dots
    i_r}$ where the $i_\ell$ in both cases run through all strictly
  increasing sequences of integers between $1$ and $n+1$.
\end{remark}

\section*{Acknowledgements}
ARG gratefully acknowledges support from the Royal Society of
  New Zealand via Marsden Grant 10-UOA-113; A\v C and MH gratefully
  acknowledge support by projects P19500--N13 and P23244-N13 of the
  ``Fonds zur F\"orderung der wissenschaftlichen For\-schung'' (FWF) and
  the hospitality of the University of Auckland. The authors are also
  grateful for the support of the meetings ``Cartan Connections,
  Geometry of Homogeneous Spaces, and Dynamics'' hosted by the Erwin
  Schr\"odinger Institute (ESI, University of Vienna) and ``The
  Geometry of Differential Equations'' hosted by the Mathematical
  Sciences Institute of the Australian National University
\def\polhk#1{\setbox0=\hbox{#1}{\ooalign{\hidewidth
  \lower1.5ex\hbox{`}\hidewidth\crcr\unhbox0}}}

\end{document}